\documentclass[12pt,english]{amsart}
\usepackage{lmodern}

\usepackage[T1]{fontenc}
\usepackage[latin9]{inputenc}
\usepackage{mathtools}
\usepackage{amstext}
\usepackage{amsthm}
\usepackage{amssymb}

\makeatletter
\numberwithin{equation}{section}
\numberwithin{figure}{section}
\usepackage{enumitem}
\theoremstyle{plain}
\newtheorem{thm}{\protect\theoremname}
\theoremstyle{plain}

\usepackage{fullpage}
\usepackage{amsfonts}
\usepackage{pdfsync}
\usepackage{tikz}
\usetikzlibrary{decorations.markings}
\usetikzlibrary{decorations.pathmorphing}

\usetikzlibrary{arrows}
\tikzset{->-/.style={decoration={ markings,mark=at position #1 with {\arrow{stealth}}},postaction={decorate}}}
\tikzset{-<-/.style={decoration={ markings,mark=at position #1 with {\arrow{stealth reversed}}},postaction={decorate}}}

\usepackage{amssymb}
\usepackage{enumitem}
\usepackage{algorithmic}
\usepackage{algorithm}

\usepackage{fancyhdr}

\usepackage{color}
\newtheorem{Theorem}{Theorem}[]
\newtheorem{Claim}{Claim}[section]

\newtheorem{Conjecture}{Conjecture}

\newtheorem{Definition}[Theorem]{Definition}
\newtheorem{lemma}{Lemma}[section]

\theoremstyle{definition}
\newcommand{\cA}{\mathcal{A}}
\newcommand{\cB}{\mathcal{B}}

\def\ol{\overline}
\def\ex{{\rm ex}}
\def\EX{{\rm EX}}
\def\cex{\overline{\rm ex}}
\def\CEX{\overline{\rm EX}}

\def\GG{\mathbb{G}}
\def\KK{\mathbb{K}}
\def\HH{\mathcal{H}}

\usepackage[normalem]{ulem}
\usepackage{babel}
\providecommand{\lemmaname}{Lemma}
\providecommand{\theoremname}{Theorem}

\begin{document}
\title{\bf {Tur\'{a}n number of three disjoint cliques of arbitrary sizes}}

\author{Alexandr Kostochka}
\address{University of Illinois at Urbana--Champaign, Urbana, IL 61801}
\email{kostochk@illinois.edu}
\thanks{Research of the first author is supported in part by NSF RTG Grant DMS-1937241.}

\author{Dadong Peng}
\address{University of Illinois at Urbana--Champaign, Urbana, IL 61801}
\email{dadongp2@illinois.edu}
\thanks{Research of the second author is supported in part by NSF RTG Grant DMS-1937241.}

\author{Zimu Xiang}
\address{Shanghai Center for Mathematical Sciences, Fudan University, Shanghai 200433, P.R.China}
\email{zimuxiang@fudan.edu.cn}

\author{Liang Zhang}
\address{Center for Combinatorics and LPMC, Nankai University, Tianjin 300071, P.R.China.}
\email{1120240010@mail.nankai.edu.cn}
\thanks{Research of the fourth author is supported  by the Fundamental and Interdisciplinary Disciplines Breakthrough Plan of the Ministry of Education of China (JYB2025XDXM207).}

\date{}

\begin{abstract} The {\it Tur\'{a}n number} $\ex(n,H)$ of a given graph $H$ is the maximum number of edges in an
$n$-vertex $H$-free graph. In this paper, for all $p\ge q\ge r\ge 2$ and $n\geq p+q+r$, we determine the exact value of $\ex(n,K_p\cup K_q\cup K_r)$ and describe all  extremal graphs.\\
\end{abstract}
\maketitle  
\section{Introduction}

For given graphs $G$ and $H$, $G$ is {\it $H$-free} if $G$ contains no subgraph isomorphic to $H$. The Tur\'{a}n number of a graph $H$, denoted by $\ex(n,H)$, is the maximum number of edges in an $n$-vertex $H$-free graph. An $n$-vertex $H$-free graph with $\ex(n,H)$ edges is called an {\em extremal graph} for $H$. Let $\EX(n,H)$ denote the family of all $n$-vertex extremal graphs for $H$. 

For integers $n$ and $p$, let $T_{n,p}$ denote the Tur\'{a}n graph, that is, the complete $p$-partite graph on $n$ vertices whose partite sets differ in size by at most one.
Let $t_{n,p}$ denote the number of edges of $T_{n,p}$. Tur\'{a}n~\cite{Turan1941} proved 

\begin{thm}[\cite{Turan1941}]\label{Turan}
For all $n\geq p\geq 2$, $\ex(n,K_p)=t_{n,p-1} $ and $\EX(n,K_p)=\{T_{n, p-1}\}$.
\end{thm}
This Tur\'{a}n's theorem generalizes Mantel's theorem~\cite{Mantel1907} stating that $\ex(n,K_3)=\lfloor\frac{n^2}{4}\rfloor$.

We follow the notation and terminology of \cite{BondyMurty1976} unless stated otherwise. For a graph $G$, let $V(G)$ and $E(G)$ denote its vertex set and edge set, respectively. The cardinality of a set $S$ is denoted by $|S|$, and $e(G)=|E(G)|$. 

The disjoint union of graphs $G$ and $H$ is denoted by $G\cup H$. The graph $G\vee H$ denotes the join of $G$ and $H$, obtained by adding all possible edges between $V(G)$ and $V(H)$.

Let $K_p$ denote the complete graph on $p$ vertices, $\ol{K}_p$ denote the edgeless graph on $p$ vertices and $S_{\ell}$ denote a star on $\ell+1$ vertices. Furthermore, let
$kK_p$ denote the disjoint union of $k$ copies of $K_p$.

Observe that a copy of $K_p$ in $G$ corresponds to an independent set of size $p$ in $\overline{G}$. Therefore, rather than directly determining $\ex(n,H)$, we first study the minimum number of edges in a graph that contains no collection of pairwise disjoint independent sets of prescribed sizes. The extremal graphs for the original problem can then be obtained by taking complements. For convenience, define $$\ol{\ex}(n,H)=\binom{n}{2}-\ex(n,H), \ol{t}_{n,p}=\binom{n}{2}-t_{n,p}\mbox{ and }\ol{\EX}(n,H)=\{\ol{G}: G\in\EX(n,H)\}.$$

The problem of determining the Tur\'{a}n number of vertex-disjoint copies of cliques was initiated by Erd\H{o}s~\cite{Erdos1962}. Subsequently, Moon~\cite{Moon1968} and Simonovits~\cite{Simonovits1968} determined $\ex(n,kK_p)$ and characterized the extremal graphs for sufficiently large $n$. In our notation, this is equivalent to determining $\ol{\ex}(n,kK_p)$ and $\ol{\EX}(n,kK_p)$.

\begin{thm}[\cite{Moon1968,Simonovits1968}]\label{Turan2}
Let $k\geq 1$ and $p\geq 2$ be integers. For sufficiently large $n$, $\ol{\ex}(n,kK_p)=\ol{t}_{n-k+1,p-1}$ and $\ol{\EX}(n,kK_p)=\{\ol{K}_{k-1}\cup\ol{T}_{n-k+1,p-1}\}$. 
\end{thm}

A natural problem is to determine the Tur\'{a}n numbers of disjoint unions of cliques for all $n$. However, exact values are known only for a limited number of cases. In 1959, Erd\H{o}s and Gallai \cite{ErdosGallai1959} determined $\ol{\ex}(n,kK_2)$ and $\ol{\EX}(n,kK_2)$ for all positive integers $n$ and $k$. More recently, Chen, Lu and Yuan \cite{CLY} determined the Tur\'{a}n number of $2K_p$ for all $n$. Subsequently, Zhang and Yin \cite{ZhangYin2024Klambda} determined the value of $\ex(n,K_p\cup K_q)$ for all $n,q$ and $p\in\{2,3\}$. This result was later extended by Hu \cite{Hu2024}, who determined $\ex(n,K_p\cup K_q)$ for all $p$ and $q$, and Luo \cite{LuoZ} characterized all extremal graphs for $K_p\cup K_q$.

For positive integers $a\le b$ and nonnegative integers $x,y,n$ satisfying $ax+by\le n$, let 
\begin{equation}\label{Kab}
  K_{a,b}(n,x,y)=xK_a\cup yK_b\cup\ol{K}_{n-ax-by}  
\end{equation}
{ Also, for each integer $z$ with $0\le z\le p-1$, let $$\KK_{3,4}(n,z)=\left\{K_{3,4}(n,x,y): x+2y=n-p-q+1\mbox{ and }x+y\le z\right\},$$ and let  $\KK_{3,4}(n)=\KK_{3,4}(n,p-1).$}

\begin{thm}[\cite{Hu2024,LuoZ}]\label{cite1}

Let $n$, $p$, $q$ be positive integers with $p>q\ge 2$ and $n\ge p+q$. Then $$\cex(n,K_p\cup K_q)=\left\{\begin{array}{lll} 3(n-p-q+1), & \mbox{if $n\le p+q+\max\{2q-p,\lfloor\frac{q}{2}\rfloor-1\}$,}\\
\ol{t}_{n,p-1},  & \mbox{if $n>p+q+\max\{2q-p,\lfloor\frac{q}{2}\rfloor-1\}$.}\end{array}\right.$$

Furthermore, if $p-1=q$, then $$\CEX(n,K_p\cup K_q)=\left\{\begin{array}{lll} \KK_{3,4}(n), & \mbox{if $n< 4q$,}\\
\{\ol{T}_{n,p-1}\},  & \mbox{if $n\ge 4q$.}\end{array}\right.$$

If $q+1\le p-1<\frac{3q}{2}$, then $$\CEX(n,K_p\cup K_q)=\left\{\begin{array}{lll} \KK_{3,4}(n), & \mbox{if $n< 3q$,}\\
\KK_{3,4}(n)\cup\{\ol{T}_{n,p-1}\}, & \mbox{if $n=3q$,}\\
\{\ol{T}_{n,p-1}\},  & \mbox{if $n>3q$.}\end{array}\right.$$

If $p-1\ge\frac{3q}{2}$, then $$\CEX(n,K_p\cup K_q)=\left\{\begin{array}{lll} \KK_{3,4}(n), & \mbox{if $n< p-1+\frac{3q}{2}$,}\\
\KK_{3,4}(n)\cup\{\ol{T}_{n,p-1}\}, & \mbox{if $n=p-1+\frac{3q}{2}$,}\\
\{\ol{T}_{n,p-1}\},  & \mbox{if $n>p-1+\frac{3q}{2}$.}\end{array}\right.$$

\end{thm}

Brualdi and Mellendorf \cite{BrualdiMellendorf1994} and independently Zhang \cite{ZhangJ} determined $\ex(kp,kK_p)$ and $\EX(kp,kK_p)$ for all $k$ and $p$. Recently, Zhang and Yin \cite{ZhangYin2024ThreeCliques} and independently Zhang \cite{ZhangJ} determined the values of $\ex(n,3K_p)$ for all $n$, and Kostochka, Peng and Zhang \cite{KostochkaPengZhang2025} determined the values of $\ex(n,4K_p)$ and $\EX(n,4K_p)$ for all $n$.

In this paper, we  determine $\cex(n,K_p\cup K_q\cup K_r)$ and $\CEX(n,K_p\cup K_q\cup K_r)$ 
for all $p\geq q\geq r\geq 2$ and $n\ge p+q+r$. We first determine the value for the case $n=p+q+r$ and the corresponding extremal graphs.  Note that the special case $p=q=r=2$ was already settled in \cite{ErdosGallai1959}, where it was shown that $\cex(6,3K_2)=5$ and $\CEX(6,3K_2)=\{K_{1,5}\}$.

\begin{thm}\label{thm11}
Let $p\ge q\ge r\ge 2$, $(p,q,r)\not=(2,2,2)$ and $n=p+q+r$. Then $$\cex(n,K_{p}\cup K_{q}\cup K_{r})=\min\{6,\overline{t}_{n,p-1}\}\quad \mbox{and }$$ $$\CEX(n,K_{p}\cup K_{q}\cup K_{r})=\begin{cases}
\{K_4\cup\ol{K}_3, K_{1,6}\}, &
\mbox{if $(p,q,r)=(3,2,2)$,}\\
\{K_4\cup\ol{K}_{n-4}, \overline{T}_{n,p-1}\}, &
\begin{array}[t]{@{}l@{}}
\mbox{if $(p,q,r)=(5,2,2)$}\\
\mbox{or $p\ge 7$ and $(q,r)=(3,2)$,}
\end{array}\\
\{\overline{T}_{n,p-1}\}, &
\mbox{if $p\ge 6$ and $q=r=2$,}\\
\{K_4\cup\ol{K}_{n-4}\}, &
\mbox{otherwise.}
\end{cases}$$
\end{thm}

{ For $p\ge q\ge r$, let $c_q=\left\{\begin{array}{lll} 0, & \mbox{if $p>q$,}\\ 
1,  & \mbox{if $p=q$,}\end{array}\right.$ and $c_r=\left\{\begin{array}{lll} 0, & \mbox{if $p>r$,}\\ 
1,  & \mbox{if $p=r$.}\end{array}\right.$} Define 
$c=c_q+c_r$.
{ Also, for each integer $z$ with $0\le z\le p-1$, let $$\KK_{5,6}(n,z)=\left\{K_{5,6}(n,x,y): 2x+3y=n-p-q-r+1\mbox{ and }x+y\le z\right\},$$ and let $\KK_{5,6}(n)=\KK_{5,6}(n,p-1).$} Furthermore, for $p=q$ and $\frac{3(p+q)}{2}-2\le n\le 2p+2q-4$,  let $$K_{3,4}^*(n)=(2p+2q-4-n)K_3\cup (n-\frac{3(p+q)}{2}+2)K_4\cup K_{1,3}.$$ For $p=q=r$ and $\frac{5(p+q+r)}{3}-4\le n\le 2p+2q+2r-6$, we let $$K_{5,6}^*(n)=(2p+2q+2r-6-n)K_5\cup (n-\frac{5(p+q+r)}{3}+4)K_6\cup K_{1,5}.$$ Our main results are as follows.

\begin{thm}
\label{pqr2}
Let $n$, $p$, $q$, and $r$ be positive integers with $p\ge q\ge r\ge 2$ and $n\ge p+q+r+1$, and set $s=n-p-q-r+1$.

1)  If $n\in [p+q+r+1, 4(p-1)+c]$, then 
\begin{equation}\label{mt1}
  \cex(n,K_p\cup K_q\cup K_r)= \min\{5s,3(s+r),\ol{t}_{n-c,p-1}\}.  
\end{equation}

2) If $n\in [4(p-1)+c, 5(p-1)+c]$, then 
\begin{equation}\label{mt2}
  \cex(n,K_p\cup K_q\cup K_r)= \min\{5s,\ol{t}_{n-c,p-1}\}.  
\end{equation}

3) If $n\geq  5(p-1)+c$, then 
$ \cex(n,K_p\cup K_q\cup K_r)= \ol{t}_{n-c,p-1}$.  
\end{thm}

{\begin{thm}\label{EXpqr2}
    Let $p\ge q\geq r\geq 2$, $n\geq p+q+r+1$ and $s=n-p-q-r+1$. Let $\mathcal{H}_0(n),\mathcal{H}_1(n),\mathcal{H}_2(n)$ be pairwise disjoint families such that $\CEX(n,K_p\cup K_q\cup K_r)=\mathcal{H}_0(n)\cup\mathcal{H}_1(n)\cup\mathcal{H}_2(n)$, where $\mathcal{H}_0(n)$ consists of graphs containing no $\ol{K}_p$, $\mathcal{H}_1(n)$ consists of graphs containing $\ol{K}_p$ and no disjoint induced $\ol{K}_p$ and $\ol{K}_q$, and $\mathcal{H}_2(n)$ consists of graphs containing disjoint induced $\ol{K}_p$ and $\ol{K}_q$. Then

    (1) If $c=0$, then $\mathcal{H}_0(n)=\{G=\overline{T}_{n,p-1}: \cex(n,K_p\cup K_q\cup K_r)=\ol{t}_{n,p-1}\}$, $$\mathcal{H}_1(n)=\{G\in{ \KK_{3,4}(n,q-1)}: \cex(n,K_p\cup K_q\cup K_r)=3(s+r)\},$$ and $$\mathcal{H}_2(n)=\{G\in{ \KK_{5,6}(n,r-1)}: \cex(n,K_p\cup K_q\cup K_r)=5s\}.$$

    (2) If $c=1$, then $\mathcal{H}_0(n)=\emptyset$, $$\begin{aligned}
    \mathcal{H}_1(n)=&\{G\in\KK_{3,4}(n): \cex(n,K_p\cup K_q\cup K_r)=3(s+r)\}\\
    &\cup\{G=K_{3,4}^*(n): \cex(n,K_p\cup K_q\cup K_r)=3(s+r) \mbox{ and }\tfrac{3(p+q)}{2}-2\le n\le 2p+2q-4\}\\
    &\cup\{G=\ol{K}_1\cup\overline{T}_{n-1,p-1}: \cex(n,K_p\cup K_q\cup K_r)=\ol{t}_{n-1,p-1}\}.
    \end{aligned}$$ and $$\mathcal{H}_2(n)=\{G\in{ \KK_{5,6}(n,r-1)}: \cex(n,K_p\cup K_q\cup K_r)=5s\}.$$

    (3) If $c=2$, then $\mathcal{H}_0(n)=\emptyset$, $\mathcal{H}_1(n)=\emptyset$ and $$\mathcal{H}_2(n)=\left\{\begin{array}{lll} \KK_{5,6}(n), & \mbox{if $n\le 5p-5$,}\\
    \KK_{5,6}(n)\cup\{K_{5,6}^*(n)\}, & \mbox{if $5p-4\le n\le 6p-6$,}\\ 
    \{\ol{K}_2\cup\overline{T}_{n-2,p-1}\},  & \mbox{$n\ge 6p-5$.}\end{array}\right.$$
\end{thm}}

{ These results extend and generalize Theorem~\ref{cite1} as well as
the result for $\cex(n,K_p\cup K_p\cup K_p)$ established in
\cite{YanZhangPeng2024,ZhangJ,ZhangYin2024ThreeCliques}. }
The key idea is  to consider graphs that already have independent sets $P$ and $Q$ of required size, and try to add an independent set $R$ to them.
The restriction $r\geq 2$ is natural, since for $n\geq p+q+1$, {$\cex(n,K_p\cup K_q\cup K_1)=\cex(n,K_p\cup K_q)$}.

{ The structure of the paper is as follows. In the next section, we 
introduce some more notation and give some definitions. In Section~\ref{sec3} we 
present Theorem~\ref{PQ} on existence of two cliques when we know that one of them exists. Section~\ref{setup} establishes Theorem~\ref{PQR}, a generalization of Theorem~\ref{PQ} that addresses another important special case of the main theorem. The subsequent three sections are devoted to proving Theorem~\ref{PQR} by considering several distinct cases. We prove Theorems~\ref{thm11}, \ref{pqr2} and \ref{EXpqr2} in Sections~\ref{help}, \ref{main} and~\ref{main2} { respectively,} and conclude the paper with some remarks in Section~\ref{remarks}.}

\section{Notation and definitions}

For a vertex $v\in V(G)$ and a subgraph $H\subseteq G$ (or a vertex subset $S\subseteq V(G)$), we write $N_H(v)$ (or $N_S(v)$) for the set of neighbors of $v$ in $H$ (or in $S$). The degree of $v$ in $G$ is denoted by
$d_G(v)=|N_G(v)|$. The minimum and maximum degrees of $G$ are denoted by $\delta(G)$ and $\Delta(G)$, respectively. We write $\ol{G}$ for the complement of $G$ and $\alpha(G)$ for its independence number. 

Given vertex subsets $V_1,V_2\subseteq V(G)$, let $E[V_1,V_2]$ denote the set of edges with one endpoint in $V_1$ and the other in $V_2$, and let $e[V_1,V_2]=|E[V_1,V_2]|$. For a vertex subset $V'\subseteq V(G)$, the subgraph induced by $V'$ is denoted by $G[V']$.

The following definition originates from~\cite{KiersteadKostochka2008}.
{We will use it throughout the proofs.}

\begin{Definition}
Given a graph $G$ and a partition $\mathcal{P}=(V_1,\dots,V_k)$ of $V(G)$, an auxiliary digraph $D$ with respect to $\mathcal{P}$ is a digraph with vertex set $V(D)=\{V_1,\dots,V_k\}$ and edge set $E(D)=\{V_iV_j:\exists x\in V_i, N_G(x)\cap V_j=\emptyset\}$.
\end{Definition}
When the partition $\mathcal{P}$ is clear, we may refer to $D$ simply as the auxiliary digraph.
For an edge $V_iV_j\in E(D)$, if $x\in V_i$ and $N_G(x)\cap V_j=\emptyset$, then we say $x$ is \emph{movable} to $V_j$.
Given $V_{i_1},V_{i_2},\dots, V_{i_\ell}\in V(D)$, if for each $V_{i_j}$ where $1\le j\le \ell-1$, there is a vertex $v_{i_j}$ that is movable to $V_{i_j+1}$, then the directed path $V_{i_1}V_{i_2}\dots V_{i_\ell}$ is an \emph{accessible} path in $D$.
For a given digraph $D$, we may fix some $V_j\in V(D)$ as a \emph{target class.}
Consequently, for $V_i\in V(D)$, if there is an accessible path from $V_i$ to $V_j$, then we say $V_i$ is \emph{accessible}.
Note that the target class $V_j$ is always trivially accessible.
Let $\cA$ be the set of accessible classes in $D$ and $\cB=V(D)\setminus\cA$.
Further, let $A=\bigcup\cA$ and $B=\bigcup\cB$.
For  vertices $u\in B $ and $v\in A$, the edge $uv$ is a \emph{solo edge}  if $v$ is the only neighbor of $u$ in its color class.
In this case, we also say that $u$ and $v$ are {\em solo neighbors}  of each other.

\section{Two cliques}\label{sec3}

\begin{thm} \label{PQ} 
Let $p\geq q\geq 2$, $n\geq p+q$ and $s=n-p-q+1$. If an $n$-vertex graph $G$ contains an induced $\ol{K}_p$ and $e(G)\le3s$,
 then one of the following holds:
 \begin{enumerate}[label=\textup{(\alph*)}, ref=\textup{\alph*}]
     \item\label{PQ:induced}$G$ contains vertex disjoint induced $\ol{K}_p$ and $\ol{K}_q$;
    \item \label{PQ:cliques}$e(G)=3s$, $G$ contains $x$ copies of $K_3$ and $y$ copies of $K_4$, where $3x+6y=3s$ and $3x+4y\leq n$, and all other components are $K_1$s;
    \item \label{PQ:star}$e(G)=3s$, $p=q$, one component of $G$ is the star $K_{1,3}$ and every other component is a $K_3$ or a $K_4$.
 \end{enumerate}
\end{thm}

\begin{proof} 
Suppose  there is a counter-example $G$ for the theorem and a set $P\subset V(G)$
 such that $G[P]=\ol{K}_p$. Then~(\ref{PQ:induced}) fails and $e(G)\leq 3s$. If $e(G)=3s$, then neither~(\ref{PQ:cliques}) nor~(\ref{PQ:star}) holds.

Let $Q$ be a maximum independent set in $G-P$ and  $q'=|Q|$. 
Let $S=V(G)-P-Q$ and $s'=|S|$. Then $1\leq q'\leq q-1$ and $s'\geq s$.

Suppose first that $s'=1$. Since $n\geq p+q$, this means $q'=q-1, s'=s=1$ and $n=p+q$. 
Thus, $e(G)\leq 3s=3$.
Choose $P$ so that
\begin{equation}\label{PQ1}
  \mbox{  $G-P$ has the fewest edges.}
\end{equation}

Note that $S\cup Q=V(G)-P$.
If $G-P$ has no edges, we are done, so assume ${ uv}\in E(G)$ for some $u,v\in S\cup Q$. 
We will use the following properties of our partition.
\begin{equation}\label{PQxy}
\parbox{0.9\linewidth}{\it (a) If $u$ or $v$ has no neighbors in $P$, then each $w\in P$ has a neighbor in $Q\cup S$.\\
(b) If $u$ or $v$ has exactly one neighbor $w$  in $P$, then $w$ has another neighbor in $Q\cup S$.}
\end{equation}
Indeed, if (\ref{PQxy}a) or (\ref{PQxy}b) does not hold, then we  switch $u$ (or $v$) with $w$, contradicting \eqref{PQ1}.

Now since $e(G)\le 3$, by (\ref{PQxy}), we must have that $|E(P,Q\cup S)|=2$ and $|P|=2$ or some $w\in P$ is adjacent to both $u$ and $v$. 

If some $w\in P$ is adjacent to both $u$ and $v$, then (b) holds for $G$, contradicting the choice of $G$.
Thus we may assume that each $w\in P$ is adjacent to at most one of $u$ and $v$ and that $|P|=2$. In this case, $p=q=2$,
 $|Q|=q'=1$, and we may assume that $u$ is the only vertex in $S$, and $v$ is the only vertex in $Q$.
Since $|E(P,Q\cup S)|=2$, by~\eqref{PQxy}(b), either $u$ or
$v$ is adjacent to both vertices in $P$.
As $e(G)=3$, we have $G=K_{1,3}$, and (c) holds for $G$,
a contradiction. So $s'\ge 2$.

We choose a triple $(G,P,Q)$ so that
\begin{enumerate}[label=\textup{(M\arabic*)}, ref=\textup{M\arabic*}]
\item \label{PQmin:n} $n:=|V(G)|$ is minimum;
\item \label{PQmin:q'} subject to \ref{PQmin:n}, $q'=|Q|$ is maximum;
\item \label{PQmin:highdeg} subject to (\ref{PQmin:n}) and (\ref{PQmin:q'}), the number of vertices in $P$ of degree at least $3$ is minimum.
\end{enumerate}

\begin{Claim}\label{PQx2}
    For each $u\in S\cup Q$,  $d_G(u)\le 2$.
\end{Claim}
\begin{proof}
Suppose the claim does not hold. Among different choices of $P$ and
 $u\in S\cup Q$ with $d(u)\geq 3$, pick such that $d(u)$ is  maximum, and modulo this, the number of $K_3$-subgraphs containing $u$  is maximum.

Consider $G'=G-u$. It  still contains $P$. If $G'$ contains vertex disjoint induced $\ol{K}_p$ and $\ol{K}_q$, then so does $G$, a contradiction. Otherwise, by the minimality of $n$, $e(G')\geq 3(s-1)$, and in the case of equality
  every component of $G'$   is a $K_3$, a $K_4$ or a $K_1$ (or a $K_{1,3}$ when $p=q$). Since $d(u)\geq 3$,   $d(u)=3$ and $e(G)=3s$.

{\bf Case 1:}
 $u$ is in a $K_4$-subgraph $F$ of $G$. Then $F$ is a component of $G$, since otherwise we can choose a vertex $v$ of a higher degree in $Q\cup S$, possibly moving $u$ into $P$. (If $v\in S\cup Q$, then $d_G(v)>d_G(u)$, contradicting the choice of $u$. If $v\in P$, switch $u$ and $v$ contradicting the choice of $u$). In this case, $F-u$ is a $K_3$-component in $G'$. Thus, if $G'$ satisfies (\ref{PQ:cliques}) or (\ref{PQ:star}) of Theorem~\ref{PQ}, then so does $G$, a contradiction. 

 Since this case is impossible and any $P$ contains at most one vertex from any $K_4$-subgraph of $G$, we conclude that
 \begin{equation}\label{NoK4}
    \mbox{$G$ does not contain $K_4$.}
 \end{equation}

{\bf Case 2:}
 $u$ is in a $K_3$-subgraph $F$ of $G$. Then $F-u$ cannot be in a $K_4$-component of $G'$ by~\eqref{NoK4}.
If $F-u$ is in a $K_{1,3}$-component, say $Y$, then for $G'$ situation (c) holds, i.e. $e(G')=3(s-1)$, $p=q$ and in view of~\eqref{NoK4} every component apart from $Y$ is a $K_3$. To have $e(G')=3(s-1)$, the number of $K_3$-components in $G'$ must be $s-2$. So, $n-1=4+3(s-2)$, and hence $s=p$.
 If $Y$ has center $v$,  then $uv\in E(G)$.
 Now $v$ is not in any independent set of size $p$ in $G'$, and hence belongs to $Q\cup S$. Since $d_G(v)=4>3=d_G(u)$, this contradicts the choice of $u$.
 
 Thus, $F-u$ is in a $K_3$-component, say $Y$ with vertex set $\{v_1,v_2,v_3\}$. By~\eqref{NoK4}, we may assume $N(u)\cap V(Y)   =\{v_1,v_2\}$.
Since $d(u)=3$, each of  $v_1,v_2$ is in two $K_3$-subgraphs of $G$ and at most one of $v_1,v_2$ can be in $P$, this contradicts the choice of $u$ and finishes Case 2.

{\bf Case 3:} Vertex $u$ is in no $K_3$-subgraph  of $G$. 
 If $G'$ has a $K_{1,3}$-component $F$, say with center $v$ and leaves $v_1,v_2,v_3$, then  by the minimality of $n$, part (c) of the theorem holds for $G'$.
This means $e(G')=3(s-1)$, $p=q$ and  by~\eqref{NoK4}
 all  components of $G'-F$ are $K_3$'s. 
So, as in Case 2, the number of $K_3$-components in $G'$ must be $s-2$,  $s=p$  and $v$ is not in any independent set of size $p$ in $G'$; thus $v\in Q\cup S$. If $uv\in E(G)$, then
  $d_G(v)=4>3=d_G(u)$, contradicting the choice of $u$. So, $uv\notin E(G)$.
For $1\leq i\leq s-2$, let the $K_3$-component $Z_i$ have vertex set $\{v_{i,1},v_{i,2},v_{i,3}\}$.
 
 By Case 2, $u$ has at most one neighbor in each $Z_i$, so we may assume $N(u)\cap Z_i\subseteq \{v_{i,3}\}$ for each $i\in [s-2]$.
 But then $\{v_{1,1},\ldots,v_{s-2,1},v_1,v_2\}$ and
 $\{v_{1,2},\ldots,v_{s-2,2},u,v\}$ are two 
  independent sets of size $p$ in $G$ contradicting the choice of $G$.

 The remaining possibility is that each component of $G'$ is a $K_3$- or a $K_1$-component and $u$ has at most one neighbor in each $K_3$-component. 
 Let the isolated vertices in $G'$ be $v_1,v_2,\ldots,v_{k_1}$ and the
$K_3$-components be $Z_1,\ldots,Z_{k_3}$ with $V(Z_i)=\{v_{i,1},v_{i,2},v_{i,3}\}$ for $i=1,\ldots,k_3$,
  so $n-1=k_1+3k_3$.
 
Since $3k_3=e(G')=e(G)-3=3s-3$ and $v(G')=k_1+3k_3$, we have $k_3=s-1$ and
   $p+q+s=n+1=k_1+3(s-1)+2$, implying 
 \begin{equation}\label{4}
     p+q-1=k_1+2(s-1).
 \end{equation}
 Since $P$ has at most one vertex in each $K_3$-component and does not contain $u$, $G$ has an independent set $Q'$ disjoint from $P$ that contains $u$ and one vertex from each $K_3$-component. Since $G$ is a counter-example to our theorem,
 $(s-1)+1=|Q'|\leq q-1$, i.e., $q\geq s+1$. Thus,~\eqref{4} implies $k_1\geq 3$ with equality only if $p=q=s+1$. 
 
 If $k_1\geq 4$, then by~\eqref{4}, 
 $$q\leq \lfloor(p+q)/2\rfloor\leq (s-1)+\lfloor(k_1+1)/2\rfloor\leq (s-1)+ k_1-2.$$
 So, since $u$ has at least $k_1-3$ non-neighbors in $K_1$-components, we can construct a $q$-element independent set $Q'$ by including into it $u$,  one non-neighbor of $u$ from each of the $s-1$ $K_3$-components and 
 $q-1-(s-1)$ isolated vertices of $G'$ not adjacent to $u$, and then construct a $p$-element independent set $P'$ by including into  it one vertex not included into $Q'$ from each of the $s-1$ $K_3$-components  and 
 $p-(s-1)$ isolated vertices of $G'$ not included into $Q'$. Note that by~\eqref{4}, there are $p-(s-1)$ isolated vertices of $G'$ not included into $Q'$. This contradicts the choice of $G$, so $k_1=3$ and $p=q=s+1$.

 Since $u$ has at most one neighbor in each $Z_i$, we may assume that $v_{1,1},\ldots,v_{s-1,1}$ are not adjacent to $u$. So if $u$ is not  adjacent to all of $v_1,v_2,v_3$, say $uv_1\notin E(G)$, then we have disjoint $(s+1)$-element independent sets $Q_1=\{v_{1,1},\ldots,v_{s-1,1},u,v_1\}$ and
$P_1=\{v_{1,2},\ldots,v_{s-1,2},v_2,v_3\}$ contradicting the choice of $G$.
Otherwise, one component of $G$ is $K_{1,3}$ (with the center $u$) and all other are $K_3$'s;
thus part~(\ref{PQ:star}) of Theorem~\ref{PQ} holds.

\end{proof}

Now let $D$ be the auxiliary digraph with respect to $(P,Q,S)$ with target class $Q$.
If $S\in \cA$, then moving witnesses along a shortest $S,Q$-path would increase $|Q|$, contradicting (\ref{PQmin:q'}). 
Thus $S\in \mathcal{B} $.
Subject to (\ref{PQmin:n})--(\ref{PQmin:highdeg}), choose the partition so that:
\begin{enumerate}[label=\textup{(M\arabic*)}, ref=\textup{M\arabic*}, start=4]
\item \label{PQmin:B} $|\mathcal B|$ is minimum.
\end{enumerate}

There are two cases for $|\mathcal B|$.

\medskip
\noindent{\bf Case 1. $|\mathcal{B} |=1$.}
Then $\mathcal{B}=\{S\}$ and $ \mathcal{A}=\{P,Q\}$.
Thus every vertex of $S$ has a neighbor in both $P$ and $Q$. 
By Claim~\ref{PQx2}, each $u\in S$ has exactly one neighbor in $P$ and exactly one neighbor in $Q$. 
Denote them by $p_u\in P$ and $q_u\in Q$.

Since $u$ has no neighbor in $P-\{p_u\}$ and $d_G(u)\le 2$, the choice in~(\ref{PQmin:highdeg}) gives $d_G(p_u)\le 2$; 
otherwise replacing $p_u$ by $u$ in $P$ would decrease the number of vertices of degree at least $3$ in $P$. 
Also, $p_u$ must have a neighbor in $Q$, for if
$N_Q(p_u)=\emptyset$, then $P-p_u+u$ is an independent set of size $p$ while $Q\cup\{p_u\}$ is an independent set larger than $Q$, contradicting~(\ref{PQmin:q'}).
Hence $p_u$ has a unique neighbor in $Q$, say $q'_u$, and $d_G(p_u)=2$.

Thus each $u\in S$ gives three edges $ up_u, uq_u,p_uq'_u$, and overall there are $3s'$ edges counted in this way.
We claim that these $3s'$ edges are distinct.
Indeed, the edges $up_u$ and $uq_u$ are distinguished by their endpoint $u$. 
If $p_uq'_u=p_vq'_v$ for distinct $u,v\in S$, then the common vertex in $P$ is adjacent to $u$, to $v$, and to a vertex of $Q$, contradicting $d_G(p_u)=2$.
Similarly, $p_u\ne p_v$ whenever $u\ne v$.
Therefore, $e(G)\ge 3s'\ge 3s$.
Since $e(G)\le 3s$, equality holds. In particular, $s'=s$, and every edge of $G$ must be one of the forms $up_u, uq_u,p_uq'_u$.

We now identify the structure. For each $u\in S$, if $q_u=q_u'$, then $G[\{u,p_u,q_u\}]\cong K_3$.
Suppose that there is some $u\in S$ with $q_u\neq q_u'$, then the component containing $u$ contains the path $q_uup_uq_u'$.
If there is a component $C$ of $G$ that is a path, and say it contains $t$ vertices in $S$, then it must contain $t$ vertices in $P$ and $t+1$ vertices in $Q$.
In particular for every $u\in C\cap S$, we must have $q_u\neq q_u'$, so $s<q\le p$, then there is some isolated vertex $w\in P$.
Now $C$ contains $3t+1$ vertices, and we may pick two sets $P_C$ and $Q_C$ such that each is an independent set of size $t+1$.
We may replace $P\cap C$ by $P_C$ and $Q\cap C$ by $Q_C$, then move $w$ from $P$ to $Q$.
Now $P$ and $Q$ are two disjoint independent sets of sizes at least $p$ and $q$ each, satisfying~(\ref{PQ:induced}), a contradiction.

Suppose that there is some cycle component $C$ of length at least four.
If it contains $t$ vertices in $S$, then it must contain $t$ vertices in each of $P$ and $Q$.
Similarly, we may find independent sets $P_C$ of size at least $t$ and $Q_C$ of size at least $t+1$.
We may replace $P\cap C$ by $P_C$ and $Q\cap C$ by $Q_C$, and $P$ and $Q$ are two disjoint independent sets of sizes at least $p$ and $q$ each, satisfying~(\ref{PQ:induced}), a contradiction.

Then each component in $G$ must be a copy of $K_3$ induced by $\{u,p_u,q_u\}$ for some $u\in S$, a contradiction to the choice of $G$.

\noindent{\bf Case 2. $|\mathcal{B} |=2$.}
Then $\mathcal{A}=\{Q\}$ and $\mathcal{B} =\{P,S\}$.
Hence every vertex of $S\cup P$ has a neighbor in $Q$, so
\begin{equation}\label{PQ:sp_q}
    e[S\cup P,Q]\ge s'+p.
\end{equation}
               
By Claim~\ref{PQx2}, each vertex of $Q$ has degree at most $2$, so $s'+p\le e[S\cup P,Q]\le 2q'$.              
Since $p\ge q\ge q'+1$, it follows that $s\le s'\le q'\le q-1\le p-1$.
If $SP\notin E(D)$, then every vertex of $S$ has a neighbor in $P$, and so $e[S,P]\ge s'$.
Together with~\eqref{PQ:sp_q}, this gives $e(G)\ge s'+p+s'\ge 3s+1$, a contradiction. 
Hence we have that $SP\in E(D)$, and say some vertex $z\in S$ has no neighbor in $P$.
If there is some vertex $w\in P$ with $d_G(w)\ge3$, then $P-w+z$ is independent and contains less vertices of degree at least $3$, contradicting (\ref{PQmin:highdeg}).
With Claim~\ref{PQx2}, we have that $\Delta(G)\le 2$.

Subject to (\ref{PQmin:n})--(\ref{PQmin:highdeg}), choose $(P,Q)$ so that $\sum_{v\in Q}d_G(v)$ is minimized.
We claim the following:
\begin{equation}\label{PQ:dxy}
\parbox{0.9\linewidth}{ If $u\in S\cup P$ has a solo neighbor $v\in Q$, then $d(u)\geq d(v)$.
}
\end{equation}

Indeed, if $d_G(u)<d_G(v)$, then we move $u$ to $Q$ and $v$ to $S$. If $u\in P$,
then we also move a witness of $SP\in E(D)$ to $P$. The new partition satisfies all
previous choices, but decreases $\sum_{v\in Q}d_G(v)$, a contradiction.

Now we assign each vertex of $S\cup P$ an initial charge of $1$, and all vertices of $Q$ and all edges charge $0$. 
Thus the total initial charge is $s'+p$.

Distribute the charges so that each edge inside $G[S\cup P]$ receives charge $1/2$ from each endpoint, and then every vertex of $S\cup P$ distributes its remaining charge equally among its neighbors in $Q$.
By~\eqref{PQ:dxy} and the fact $\Delta(G)\le2$, every vertex of $Q$ receives final charge at most $1$.
Therefore $s'+p\le e(G[S\cup P])+q'$.

With~\eqref{PQ:sp_q}, we have that
$$
e(G)\ge e(G[S\cup P])+e[S\cup P,Q]\ge s'+p-q'+(s'+p)\ge (s+1)+(s'+p)\ge 3s+1,
$$
a contradiction.
\end{proof}

\section{Setup for $3$ cliques}\label{setup}
In  this section { we set up a proof of Theorem~\ref{PQR} below, and in the following three sections we prove it.}

\begin{thm} 
\label{PQR} Let $p,q\geq  r\geq 2$, $n\geq p+r+q$ and $s=n-p-q-r+1$. If an $n$-vertex graph $G$ contains  vertex disjoint  induced $\ol{K}_p$ and $\ol{K}_q$ and $e(G)\leq 5s$,
then one of the following holds:
 \begin{enumerate}[label=\textup{(\alph*)}, ref=\textup{\alph*}]
\item \label{PQR:induced} $G$ contains  vertex disjoint  induced $\ol{K}_p$, $\ol{K}_q$ and $\ol{K}_r$,
\item \label{PQR:cliques} $e(G)=5s$ and $G$ contains $x$ disjoint copies of $K_5$ and $y$ disjoint copies of $K_6$, where $10x+15y=5s$ and $5x+6y\leq n$, 
\item \label{PQR:star} $e(G)=5s$, $p=q=r$, one component of $G$ is the star $K_{1,5}$ and every other component is a $K_5$ or a $K_6$.
\end{enumerate}
\end{thm}
Note that  $p$ and $q$ are symmetric. 
Suppose there is a counter-example $G$ for the theorem and disjoint sets $P,Q\subset V(G)$
are such that $G[P]=\ol{K}_p$ and $G[Q]=\ol{K}_q$. Then~(\ref{PQR:induced}) fails and $e(G)\leq 5s$. If $e(G)=5s$, then neither~(\ref{PQR:cliques}) nor~(\ref{PQR:star}) holds.

Let $R$ be a maximum independent set in $G-P-Q$ and  $r'=|R|$. 
Let $S=V(G)-P-Q-R$ and $s'=|S|$. Then $1\leq r'\leq r-1$ and $s'\geq s$.

Suppose first that $s'=1$. Since $n\geq p+q+r$, this means $r'=r-1, s'=s$ and $n=p+q+r$. Choose $P$ and $Q$ so that
\begin{enumerate}[label=\textup{(Q\arabic*)}, ref=\textup{Q\arabic*}]
\item \label{PQR1} {\it $G-P-Q$ has the fewest edges,}
\item \label{PQR2} {\it  modulo~$($\ref{PQR1}$)$, $\sum_{v\in S\cup R}d(v)$ is minimum.}
\end{enumerate}
Note that $S\cup R=V(G)-P-Q$.
If $G-P-Q$ has no edges, we are done, so assume $xy\in E(G)$ for some $x,y\in S\cup R$. 
We will use the following properties of our partition.
\begin{equation}\label{01}
\parbox{0.9\linewidth}{\it (A) If $x$ (or $y$) has no neighbors  in some $J\in \{P,Q\}$, then each $w\in J$ has a neighbor in $R\cup S$; and\\
(B) If $x$ (or $y$) has exactly one neighbor $w$  in some $J\in \{P,Q\}$,
 then $w$ has another neighbor in $R\cup S$.}
\end{equation}
Indeed, if (A) or (B) does not hold, then we switch $x$ (or $y$) with $w$, contradicting (\ref{PQR1}).

Now~\eqref{01} yields
\begin{equation}\label{02}
\parbox{0.9\linewidth}{\it For each $J\in \{P,Q\}$, $e[J,R\cup S]\geq 2$ with equality only if $|J|=2$ or some $w\in J$ is adjacent to both $x$ and $y$.}
\end{equation}

Since
$e(G)\leq 5$, \eqref{02} implies that $e(G)= 5$ and $e[P,Q]=0$. If 
each $J\in \{P,Q\}$ has a vertex $w_J$ adjacent to two vertices in $R\cup S$, then we move $w_P$ to $Q$ and an isolated $v\in Q$ to $P$, obtaining a new partition in which $R\cup S$ has no neighbors in $P-w_P+v$, contradicting~\eqref{01}(A). Thus by symmetry we may assume that each $w\in P$ has at most one neighbor in $R\cup S$. Then by~\eqref{02},
$|P|=2$, say $P=\{w,w'\}$, and hence $|R|=1$, namely $R\cup S=\{x,y\}$.  Also, if $w$ and $w'$ are adjacent to different neighbors in $R\cup S$, that would contradict~\eqref{01}(B).
So, we may assume that $xw,xw'\in E(G)$.

Now, if $Q$ contains a vertex $z$ adjacent to $x$ and $y$, then we can switch $z$ with $x$ and decrease the sum of degrees of vertices in { $S\cup R$}, contradicting{~\eqref{PQR2}}. Therefore, 
by \eqref{02}, $|Q|=2$, say $Q=\{z,z'\}$, and $z$ and $z'$ have the same neighbor in $R\cup S$. If this neighbor is $y$, then the sets $P'=\{w,y\}$, $Q'=\{z,x\}$ and $R'=\{w',z'\}$ are independent, a contradiction. Otherwise, $G=K_{1,5}$, as claimed.
 So $s'\geq 2$.

We choose  a quadruple $(G,P,Q,R)$ as above so that
\begin{enumerate}[label=\textup{(N\arabic*)}, ref=\textup{N\arabic*}]
\item \label{PQRmin:n} {\it $n:=|V(G)|$ is minimum;}
\item \label{PQRmin:r'} {\it  modulo (\ref{PQRmin:n}), $r'$ is maximum;}
\item \label{PQRmin:highdeg}  {\it modulo (\ref{PQRmin:n}) and (\ref{PQRmin:r'}), the number of vertices in $P\cup Q$ of degree at least $5$ is minimum.}
\end{enumerate}

{
\begin{Claim}
\label{d4}
      For every $x\in S\cup R$, $d(x)\leq 4$.
\end{Claim}

\begin{proof} Suppose the claim does not hold. Among different choices of $P,Q$ and
 $x\in S\cup R$ with $d(x)\geq 5$, pick such that $d(x)$ is  maximum, and modulo this, the number of $K_5$-subgraphs containing $x$  is maximum.

Consider $G'=G-x$. It  still contains $P$ and $Q$. If $G'$ contains  vertex disjoint  induced $\ol{K}_p,\ol{K}_q$ and $\ol{K}_r$, then so does $G$, a contradiction. Otherwise, by the minimality of $n$, $e(G')\geq 5(s-1)$, and in the case of equality
  every component of $G'$   is a $K_5$, a $K_6$ or a $K_1$ (or a $K_{1,5}$ when $p=q=r$). Since $d(x)\geq 5$, the latter holds and $d(x)=5$.

{\bf Case 1:}
 $x$ is in a $K_6$-subgraph $F$ of $G$. Then $F$ is a component of $G$, since otherwise we can choose a vertex of a higher degree in $R\cup S$, possibly moving $x$ into $P\cup Q$. In this case, $F-x$ is a $K_5$-component in $G'$. Thus, if $G'$ satisfies statement (b) or (c) of our theorem, then so does $G$, a contradiction. 

 Since Case 1 is impossible and any $P\cup Q$ contains at most two vertices from any $K_6$-subgraph of $G$, we conclude that
 \begin{equation}\label{NoK6}
    \mbox{$G$ does not contain $K_6$.} 
 \end{equation}

{\bf Case 2:}
 $x$ is in a $K_5$-subgraph $F$ of $G$. Then $F-x$ cannot be in a $K_6$-component of $G'$ by~\eqref{NoK6}, and in a $K_{1,5}$-component by the structure. So, it must be in a $K_5$-component, which implies that there is $y\in V(G)-V(F)$ adjacent to all vertices in $F-x$. Since each vertex in $F-x$ is in two $K_5$-subgraphs of $G$ and
 only two  of them can be in $P\cup Q$, this contradicts the choice of $x$.

{\bf Case 3:} For any choice of $P,Q$ and $x\in V(G)\setminus (P\cup Q)$ with $d(x)=5$,
 $x$ is in no $K_5$-subgraph  of $G$. Let $x$ and $G'$ be chosen as above.
 If $G'$ has a $K_{1,5}$-component $F$, say with center $c$ and leaves $y_1,\ldots,y_5$, then $p=q=r$ and by the minimality of $n$ and~\eqref{NoK6}, all other components are $K_5$'s. In particular, if $G'$ has $k$
 $K_5$-components, then $p=q=r=k+2$. 
 It follows that we can include any two vertices from each copy of $K_5$ into $P\cup Q$, thus by Case 2, $x$ has no neighbors in $K_5$'s, and $c$  is in neither $P$ nor $Q$. Therefore $xc\notin E(G)$. But then we can construct independent sets $P,Q$ and $R$ by including into $P$ $\{y_1,y_2\}$, into $Q$ $\{y_3,y_4\}$, into $R$ $\{c,x\}$, and then add to each of these sets  one vertex from each of the $k$ $K_5$-components. This contradicts the choice of $G$.

 The remaining possibility is that each component of $G'$ is a $K_5$- or a $K_1$-component and that $x$ is not adjacent to $K_5$-components. It follows that one component of $G$ is a $K_{1,5}$, some components are $K_1$'s (say, $k_1$ of them), and some components are $K_5$'s (say, $k_5$ of them). 
 Let $G''$ be obtained from $G$ by deleting all vertices of $K_5$-components, and let $n''=|V(G'')|$.
 Then $n=5k_5+k_1+6$, $n''=k_1+6$ and $e(G)=10k_5+5$, hence $s=2k_5+1$
  and $p+q+r=3k_5+k_1+6$. Since $G'$ contains disjoint independent sets $P$ and $Q$ each of which has at most one vertex in each $K_5$-component, we can find a disjoint from $P\cup Q$ independent set containing $x$,  and one vertex from each $K_5$-component.
  Since $G$ is a counter-example to the theorem, this yields 
  \begin{equation}\label{r-k}
     r\geq 2+k_5.
 \end{equation}
   It follows that $V(G'')$
 cannot be partitioned into three independent sets of sizes $p-k_5$, $q-k_5$ and $r-k_5$, which by the structure of $G''$ means that
 \begin{equation}\label{k1}
     1+k_1<r-k_5\leq \lfloor (n''/3)\rfloor=\lfloor (k_1/3)\rfloor+2.
 \end{equation}
 By~\eqref{k1}, $k_1=0$, $r-k_5\leq 2$ and $n''=6$. Together with~\eqref{r-k} and $p,q\geq r$, we get
 $r-k_5=q-k_5=p-k_5=2$, thus~(\ref{PQR:star}) of Theorem~\ref{PQR} holds.
\end{proof}}

Similarly, we obtain the following properties:
\begin{equation}\label{11'}
\parbox{0.9\linewidth}{ (i) If $x\in R\cup S$ has no neighbor in some $J\in\{P,Q\}$, then $d(v)\leq 4$ for every $v\in J$;\\
(ii) If $x\in S\cup R$ has a unique neighbor $x'$ in some $J\in \{P,Q\}$, then 
$d(x')\leq 4$.}
\end{equation}

With the partition $(P,Q,R,S)$ of $V(G)$, let $D$ be the corresponding auxiliary digraph, and fix $R$ to be the target class, so $R$ is accessible by default. 
Let $\cA$ denote the set of accessible vertices in $D$ and let $\cB=V(D)\setminus \cA$. 
If $S$ is accessible, then we  move the witnesses along a shortest $S,R$-path in $D$ and get more vertices in $R$, contradicting (\ref{PQRmin:r'}). Thus, $S\in\cB$. 
Among the quadruples $(G,P,Q,R)$ satisfying  
(\ref{PQRmin:n})-(\ref{PQRmin:highdeg}), choose one such that
\begin{enumerate}[label=\textup{(N\arabic*)}, ref=\textup{N\arabic*}, start=4]
\item \label{PQRmin:B}  {\it $|\mathcal B|$ is minimum.}
\end{enumerate}

By above, $|\cB|\in \{1,2,3\}$. In each of the next 3 sections, we will consider one of these possibilities.

\section{Case of $|\cB|=1$}\label{B1}

In this section, $\cB=\{S\}$, $\cA=\{P,Q,R\}$ and each $x\in S$ has at least one neighbor in each $V'\in\cA$. By Claim~\ref{d4}, $\Delta(G[S])\le 1$. Thus $s\le s'\le 2r-2$. By Tur\'{a}n's Theorem, 
\begin{equation}\label{3.1}
e(G[S])\ge\ol{t}_{s',r-1}=\left\{\begin{array}{lll} 0, & \mbox{if $0\le s'\le r-1$,}\\ s'-r+1, & \mbox{if $r\le s'\le 2r-2.$}\end{array}\right.
\end{equation}

\begin{Claim}\label{claim3.1} If $v\in P$ (resp. $v\in Q$) is a solo neighbor of a vertex $u\in S$ and $QR\in E(D)$ (resp. $PR\in E(D)$), then $N_{Q}(v)\not=\emptyset$ (resp. $N_{P}(v)\not=\emptyset$) and $N_{R}(v)\not=\emptyset$.
\end{Claim}
\begin{proof} By symmetry, we may assume that $v\in P$ and $QR\in E(D)$. If $N_{Q}(v)=\emptyset$, then we move $u$ to $P$, move $v$ to $Q$ and move the witness along $QR$.  Thus we get more vertices in $R$, contradicting~\eqref{PQRmin:r'}.

If $N_{R}(v)=\emptyset$, then we move $u$ to $P$, move $v$ to $R$. Again, we obtain more vertices in $R$, contradicting~\eqref{PQRmin:r'}.\end{proof}

\begin{Claim}\label{claim3.2} No vertex $v\in V'\in\mathcal{A}$ is the common solo neighbor of two non-adjacent vertices in $S$.
\end{Claim}
\begin{proof} Assume that $v\in V'\in\mathcal{A}$ is the common solo neighbor of two non-adjacent vertices $u_1,u_2\in S$. If $V'\not= R$, then $D$ contains a $V',R$-path. Let $v'\in V'$ be a vertex with no neighbor in the successor of $V'$ on $V',R$-path. By Claim~\ref{claim3.1}, $v\not=v'$. We move the witnesses along $V',R$-path, then move $u_1,u_2$ to $V'$ and move $v$ to $S$. Now we get more vertices in $R$, contradicting~\eqref{PQRmin:r'}.\end{proof}

\begin{Claim}\label{claim3.3} Every vertex in $V'\in\{P,Q\}$ that is the solo neighbor of a vertex in $S$ has neighbors in each class of $\mathcal{A}\setminus\{V'\}$.
\end{Claim}
\begin{proof} By symmetry, we let $V'=P$ and $v\in P$ be a solo neighbor of $u\in S$. If $QR\not\in E(D)$, then $e[Q,R]\ge q$ and $PR,QP\in E(D)$. Let $v'\in Q$ be a solo neighbor of $u'\in S$. By Claim~\ref{claim3.1}, $N_{P}(v')\not=\emptyset$. Recall that every vertex in $S$ has at least one neighbor in each class in $\mathcal{A}$. Assume that there are $x$ vertices in $S$ that have a solo neighbor in $Q$, and $s'-x$ vertices in $S$ that have at least two neighbors in $Q$. {Since} $\Delta(G[S])\le 1$, by Claim~\ref{claim3.2}, there are at least $\lceil\frac{x}{2}\rceil$ vertices in $S$ that have the distinct unique neighbors in $Q$. Thus 
$$e[S,Q]+ e[Q,P]\ge x+2(s'-x)+\lceil\frac{x}{2}\rceil=2s'-\lfloor\frac{x}{2}\rfloor\ge 2s'-\frac{x}{2}\ge\frac{3s'}{2}.$$ By Claim~\ref{claim3.1}, $N_{R}(v)\not=\emptyset$. Similarly, we obtain $e[S,P]+e[P,R]\ge\frac{3s'}{2}$. Thus \begin{align*}
e(G)&=e(G[S])+(e[S,P]+ e[P,R])+(e[S,Q]+ e[Q,P])+e[S,R]+e[Q,R]\\
&\ge\ol{t}_{s',r-1}+\frac{3s'}{2}+\frac{3s'}{2}+s'+q\\
&=4s'+q+\ol{t}_{s',r-1}
\end{align*}
{ Since}  $q\ge r$, by~\eqref{3.1}, $e(G)>5s$. This contradiction to $e(G)\le 5s$ yields $QR\in E(D)$. By Claim~\ref{claim3.1}, we are done.\end{proof}

We do the following discharging. At the start, each $v\in V(G)$ has charge $ch(v)=0$ and each $e\in E(G)$ has charge $ch(e)=1$. So, $\sum_{x\in V(G)\cup E(G)}ch(x)=|E(G)|\le 5s$. Now, 

Step 1.\ \ Every edge $xy\in E(G)$ such that  $x\in S,y\not\in S$ gives charge $1$ to $x$.

Step 2.\ \ Every edge $xy\in E(G)$ such that  $x,y\in S$ gives charge $1/2$ to $x$ and $1/2$ to $y$.

Step 3.\ \ Every edge $xy\in E(G)$ such that  $x\in P\cup Q$ and $y\in R$ gives charge $1$ to $x$.

Step 4.\ \ Every edge $xy\in E(G)$ such that $x,y\in P\cup Q$ gives charge 1/2 to $x$ and 1/2 to $y$.

Step 5.\ \ Every vertex $v$ in $P\cup Q\cup R$ distributes its charge equally between {its solo neighbors in $S$.}

Let the charge of each $x\in V(G)\cup E(G)$ after Step $j$ be denoted by $ch_j(x)$. Then $ch_4(e)= 0$ for every $e\in E(G)$ and $ch_5(v)\geq 0$ for every $v\in P\cup Q\cup R$. 
Hence
\begin{equation}\label{3.2}
    e(G)\ge\sum_{v\in S}ch_5(v).
\end{equation}
We claim that 
\begin{equation}\label{3.3}
    \parbox{0.9\linewidth}{ { (i) $ch_5(v)=5$ and $d_S(v)=1$ for each $v\in S$; and\\
    (ii) The solo neighbors of $v$ in $P,Q$ are also solo neighbors of its neighbor in $S$.}}
\end{equation}
By~Claim~\ref{d4}, $d(v)\leq 4$ for every $v\in S$. Recall that $v\in S$ has at least one neighbor in each $V'\in\cA$. Suppose that $d_{S}(v)=0$. 
By Claim~\ref{claim3.2}, solo neighbors of $v$ are not solo neighbors of other vertices in $S$. Assume that $v$ has exactly $x$ solo neighbors in $P\cup Q$. Since $d(v)\leq 4$,  $x\ge 1$. Then $ch_1(v)\geq x+(2-x)\cdot 2+1$. Let $u\in P$ (resp. $u\in Q$) be a solo neighbor of $v$. By Claim~\ref{claim3.3}, $u$ has neighbors in each class in $\{Q,R\}$ (resp. $\{P,R\}$). Hence $ch_4(u)\geq\frac{1}{2}+1=\frac{3}{2}$. In Step 5, each solo neighbor will send all its charge to $v$. Therefore, $ch_5(v)\ge x\cdot(1+\frac{3}{2})+(2-x)\cdot 2+1=5+\frac{x}{2}>5$.

If $d_{S}(v)=1$, we let $N_{S}(v)=\{v'\}$. Moreover, let $N_{P}(v)=\{u_1\}$, $N_{Q}(v)=\{u_2\}$ and $N_{R}(v)=\{u_3\}$. Then $ch_2(v)=3+\frac{1}{2}$. By Claim~\ref{claim3.3}, $ch_4(u_i)\ge\frac{1}{2}+1=\frac{3}{2}$ for $1\le i\le 2$. By Claim~\ref{claim3.2}, $u_i$ ($1\le i\le 2$) can be a solo neighbor of only $v$ and $v'$. In Step 5, each $u_i$ ($1\le i\le 2$) sends at least half of its charge to $v$. Therefore, $ch_5(v)\ge 3+\frac{1}{2}+2\cdot\frac{1}{2}\cdot\frac{3}{2}=5$ and $ch_5(v)=5$ only if $u_1,u_2$ are the common solo neighbors of $v$ and $v'$. 

{ By~\eqref{3.2} and above discussion, $e(G)\ge\sum_{v\in S}ch_5(v)\ge 5s'\ge 5s$, with equality only if $s'=s$ and $d_S(v)=1$ for each $v\in S$. Since $e(G)\le 5s$, this proves~\eqref{3.3}. Moreover, we conclude that 
\begin{equation}\label{3.4}
    \mbox{$s$ is divisible by $2$ and $G[S]=\frac{s}{2}K_2$.}
\end{equation}

For $v\in S$, let $N_{S}(v)=\{v'\}$, $N_{P}(v)=\{u_p\}$ and $N_{Q}(v)=\{u_q\}$. By~\eqref{3.3}(ii), $N_{P}(v')=\{u_p\}$ and $N_{Q}(v')=\{u_q\}$. Suppose that $u_pu_q\notin E(G)$. By~\eqref{11'}, $d(u_p)\le 4$. By Claim~\ref{claim3.3}, we may assume that $N_{Q}(u_p)=\{u_q'\}$ and $u_q'\not=u_q$. We switch $v'$ with $u_p$. Then $N_{Q}(v)\not=N_{Q}(u_p)$, contradicting~\eqref{3.3}(ii). Thus $$G[\{v, v', u_p, u_q\}]=K_4.$$ Note that each $x\in \{v, v', u_p, u_q\}$ has a neighbor in $R$. By $s\le 2r-2$ and~\eqref{3.4}, every component of $G[S\cup P\cup Q]$ is either a $K_4$ or a $K_1$, and every such $K_1$ is also an isolated vertex of $G$. We claim the following.
\begin{equation}\label{3.5}
    \parbox{0.9\linewidth}{For each $K_4$, its neighbors in $R$ are disjoint from the neighbors in $R$ of every other $K_4$.}
\end{equation}
Indeed, if~\eqref{3.5} does not hold, then we could rearrange the vertices of the $K_4$'s so that a common neighbor becomes the solo neighbor of two non-adjacent vertices in $S$ contradicting Claim~\ref{claim3.2}.

If $|N_R(\{v, v', u_p, u_q\})|\ge 2$, choose vertices $u_r,u_r'\in N_R(\{v, v', u_p, u_q\})$ such that $d(u_r)$ is maximum and $d(u_r')$ is minimum. Then we may assume that $u_ru_p, u_r'v'\not\in E(G)$. Let $u_p^0\in P$ be an isolated vertex. We move $v',u_p^0$ to $R$, and move $u_r$ to $P$. Then $P-u_p^0+u_r$ and $(R-u_r)\cup\{v',u_p^0\}$ are also independent sets, contradicting~\eqref{PQRmin:r'}. Thus $|N_R(\{v, v', u_p, u_q\})|=1$. By~\eqref{3.4} and~\eqref{3.5}, $G=\frac{s}{2}K_5\cup (n-\frac{5s}{2})K_1$. Hence $G$ satisfies statement (b) of our theorem, a contradiction. 

}

\section{Case of $|\cB|=2$}\label{B2}

In this section $\cA=\{X,R\}$ and $\cB=\{S,Y\}$, where $\{X,Y\}=\{P,Q\}$. 
Let $y=|Y|$, so $y\in\{p,q\}$ and $y\ge r>r'$. Let $B=\bigcup_{W\in \cB} W$.
By the definition of accessible, each $v\in B$ has at least one neighbor in each $V'\in \cA$. 
{ By Claim~\ref{d4}, for every $u\in S$,
$$d_{G[B]}(u)=d_G(u)-|N_{X\cup R}(u)|\le 2.$$}
If $G[B]$ contains  vertex disjoint  induced { $\ol{K}_{|Y|}$ and $\ol{K}_r$}, 
 then these subgraphs together with $X$ form  vertex disjoint  induced $\ol{K}_p, \ol{K}_q$ and $\ol{K}_r$, contradicting the choice of $G$.
Thus by Theorem~\ref{PQ}, $e(G[B])\ge 3(s'-r')$.

 Let $\cB'$ denote the set of classes in $\cB$ not reachable from $S$.
Let $B'=\bigcup_{W\in \cB'} W$. 
Note that if $SY\in E(D)$, then $B'=\emptyset$, otherwise { $ B'=Y$.}
Since $X$ is accessible,  $XR\in E(D)$, and we have the following claims similar to  the ones in Section~\ref{B1}.

{
\begin{Claim}\label{claim4} If $v\in X$ is the solo neighbor of a vertex $u\in B\setminus B'$, then  $N_{R}(v)\not=\emptyset$.
\end{Claim}
\begin{proof} If  $u\in S$, then we may move $v$ to $R$ and $u$ to $X$, contradicting the choice of $r'$.
If $u\in Y$, then $SY\in E(D)$, and there is some $w\in S$ witnessing the directed edge $SY$.
Then we may move $w$ to $Y$, $u$ to $X$ and $v$ to $R$, again contradicting the choice of $r'$.
\end{proof}
}

\begin{Claim}\label{kx}
    No vertex $v\in X$ is a solo neighbor of two non-adjacent vertices of $B\setminus B'$.
\end{Claim}

\begin{proof}
    Suppose that some $v\in X$ is a solo neighbor of non-adjacent vertices $u,w\in B\setminus B'$.
    By Claim~\ref{claim4}, $v$ has a neighbor in $R$, so there is some $v'\in X\setminus\{v\}$ that has no neighbor in $R$.

    If at least one of $u,w$ is in $S$, say $u\in S$, then we may switch $u$ and $v$.
    Now $w$ is movable to the new $X$, and since $X$ is still accessible, the class containing $w$ is now accessible. This contradicts the minimality of $|\cB|$ in~\eqref{PQRmin:B}. 
    If both $u,w$ are in $Y$, then $Y$ is reachable from $S$, so we may fix a vertex $s_0\in S$ movable to $Y$.
    { Now we move $s_0$ to $Y$,  $u$ to $X$ and $v$ to $S$.}
    Again $w$ is movable to the new $X$, and now $Y$ is
accessible, contradicting~\eqref{PQRmin:B}. \end{proof}

\textbf{Case 1: $|\cB'|=0$.}
We first show that $e(Y,X)+e(X,R)\ge 2y$.
Let $Y_1=\{v\in Y: |N_X(v)|=1\}$. 
For $x\in X$, put $Y_x=\{v\in Y_1:N_X(v)=\{x\}\}$. 
By Claim~\ref{kx} and the facts that $Y$ is independent and $Y\subseteq B\setminus B'$, each $Y_x$ has size at most $1$. 
By Claim~\ref{claim4}, whenever $Y_x\neq\emptyset$, the vertex $x$ has a neighbor in $R$. 
Now the vertices of $Y\setminus Y_1$ send at least two edges to $X$, and each vertex of $Y_1$ contributes one edge to $X$ together with a distinct edge from its solo neighbor in $X$ to $R$.
Thus we have that $e(Y,X)+e(X,R)\ge 2y$.

Also, every vertex of $Y$ has a neighbor in $R$, so $e(Y,R)\ge y$. 
Since every vertex of $S$ has a neighbor in both $X$ and $R$, we have $e(S,X\cup R)\ge 2s'$. 
By Theorem~\ref{PQ}, $e(G[B])\ge 3(s'-r')$. 
Then we obtain
$$
\begin{aligned}
e(G)&\ge e(G[B])+e(S,X\cup R)+e(Y,X)+e(Y,R)+e(X,R)\\
&\ge 3(s'-r')+2s'+2y+y\\
&=5s'+3y-3r'>5s,
\end{aligned}
$$
where the last inequality follows from $s'\ge s$ and $y\ge r>r'$. 
This contradicts $e(G)\le 5s$.

\textbf{Case 2: $\cB'=\{Y\}$.} Thus $B\setminus B'=S$.
Since $Y$ is not reachable from $S$, every vertex of $S$ has a neighbor in $Y$, so $e(S,Y)\ge s'$. 
Also every vertex of $S$ has a neighbor in $R$, and every vertex of $Y$ has a neighbor in both $X$
and $R$.

Now we show that $e(S,X)+e(X,R)\ge\frac{3s'}{2}.$
Let $S_1=\{v\in S: |N_X(v)|=1\}$. 
For $x\in X$, put $S_x=\{v\in S_1:N_X(v)=\{x\}\}$. 
By Claim~\ref{kx}, each nonempty $S_x$ is a clique in $G[S]$. 
Since $Y$ is not reachable from $S$, we have $\Delta(G[S])\le 1$, and hence $|S_x|\le 2$.
By Claim~\ref{claim4}, every nonempty $S_x$ contributes an edge from $x$ to $R$. 
Therefore, if $k=|S_1|$, then
$$
e(S,X)+e(X,R)\ge 2(s'-k)+k+\left\lceil\frac{k}{2}\right\rceil
\ge \frac{3s'}{2},
$$
as desired.

Since $\Delta(G[S])\le 1$ and $\alpha(G[S])\le r-1$, we have
$s'\le 2r-2$, and Tur\'an's theorem gives
$$
e(G[S])\ge \ol t_{s',r-1}
=
\begin{cases}
0, &\mbox{if }0\le s'\le r-1,\\
s'-r+1, &\mbox{if }r\le s'\le 2r-2.
\end{cases}
$$
Now we have that
$$
\begin{aligned}
e(G)&\ge e(G[S])+e(S,X\cup Y\cup R)+e(X,R)+e(Y,X\cup R)\\
&\ge \ol t_{s',r-1}+\frac{7s'}{2}+2y.
\end{aligned}
$$
If $s'\le r-1$, then, since $y\ge r\ge s'+1$,
$$
\ol t_{s',r-1}+\frac{7s'}{2}+2y
\ge \frac{7s'}{2}+2(s'+1)>5s'\ge 5s.
$$
If $r\le s'\le 2r-2$, then
$$
\ol t_{s',r-1}+\frac{7s'}{2}+2y
\ge s'-r+1+\frac{7s'}{2}+2r
=\frac{9s'}{2}+r+1>5s'\ge 5s,
$$
In both cases $e(G)>5s$, contradicting $e(G)\le 5s$.
{ In particular we don't have extremal graphs in this case.}

\section{Case of $|\cB|=3$}\label{B3}

In this section, $|\cB|=3$, thus
 $\cB=\{P,Q,S\}$ and each $x\in S\cup P\cup Q$ has a neighbor in $R$. So 
\begin{equation}\label{15'}
 |E(S\cup P\cup Q,R)|\geq s'+p+q,
 \end{equation}
 and by Claim~\ref{d4}, $s'+p+q\leq 4r'$. Since $p,q\geq r\geq r'+1$, this yields
\begin{equation}\label{12'}
 s\leq s'\leq 4r'-p-q\leq 2r'-2\leq 2(r-2).
\end{equation}

Let $\cB'$ denote the set of sets in $\cB$ not reachable from $S$. Modulo (\ref{PQRmin:n})--(\ref{PQRmin:B}) we choose $(P,Q,R)$ so that
\begin{enumerate}[label=\textup{(N\arabic*)}, ref=\textup{N\arabic*}, start=5]
\item \label{PQRmin:cB'}  {\it  $|\cB'|$ is minimum.}
\end{enumerate}
 
Since $S$ is reachable from itself, we have three cases.

\bigskip
{\bf Case 1:} $\cB'=\emptyset$. Modulo (\ref{PQRmin:n})--(\ref{PQRmin:cB'}), we choose $(P,Q,R)$ so that
\begin{enumerate}[label=\textup{(N\arabic*)}, ref=\textup{N\arabic*}, start=6]
\item \label{PQRmin:sumd}  {\it $\sum_{y\in R}d(y)$ is minimum.}
\end{enumerate}

Similarly to~\eqref{PQ:dxy} (and with a similar proof) we now show
 \begin{equation}\label{PQR:dxy}
\parbox{0.9\linewidth}{ If $x\in S\cup P\cup Q$ has a solo neighbor $y$ in $R$, then $d(x)\geq d(y)$.
}
\end{equation}
Indeed, if $d(x)<d(y)$, then we move $x$ to $R$ and $y$ to $S$. Now, if $x$ was originally not in $S$, then by symmetry we may assume  $x\in Q$, 
consider a shortest $S,Q$-path in $D$ and move the witness(es) of the edges along the path. The new partition satisfies (\ref{PQRmin:n})--(\ref{PQRmin:cB'}), but $\sum_{y\in R}d(y)$ decreases, contradicting (\ref{PQRmin:sumd}). This proves~\eqref{PQR:dxy}.

We do the following discharging. At the start, the {\em charge} $ch(x)$ is $1$ when $x\in S\cup P\cup Q$ and is $0$ when $x\in R\cup E(G)$. Thus $\sum_{x\in V(G)\cup E(G)}ch(x)=s'+p+q$. Now,

Step 1: each edge $uv\in E(G[S\cup P\cup Q])$ gets $1/6$ from each of its two ends;

Step 2: each $u\in S\cup P\cup Q$ shares the remaining charge equally among its neighbors in $R$.

Let $ch^*(x)$ denote the resulting charge for $x\in V(G)\cup E(G)$. We claim that
 \begin{equation}\label{14'}
\mbox { $ch^*(v)\leq 2$ for every $v\in R$.
}
\end{equation}
Indeed, if $d(v)\leq 2$, this is trivial. If $d(v)=3$ and $uv\in E(G)$,
then either $v$ is not a solo neighbor of $u$ and $u$ gives to $v$ at most
$1/2$, or $v$ is  a solo neighbor of $u$ and by~\eqref{PQR:dxy},
$u$ gives to the incident edges at least $2(1/6)$ and to $v$ at most $1-1/3=2/3$. In all cases, $v$ gets at most $3(2/3)=2$.

Suppose now that $d(v)=4$. If $v$ is not a solo neighbor of $u$, then $u$ gives to $v$ at most $1/2$. Otherwise~\eqref{PQR:dxy} yields $d(u)\geq 4$, so $u$ gives to the incident edges at least $3(1/6)=1/2$ and to $v$  $1-1/2=1/2$. Anyway, $v$ gets at most $4(1/2)=2$.
 This proves~\eqref{14'}.

So, by the discharging rules and~\eqref{14'}, $s'+p+q\leq \frac{1}{3}|E(G[S\cup P\cup Q])|+2r'$ implying $e(G[S\cup P\cup Q])\geq 3(s'+p+q-2r')\geq 3(s+2)$.
Together with~\eqref{15'} and~\eqref{12'}, we get 
$$e(G)\geq 3(s+1)+(s'+p+q)\geq 4s+3+2r\geq 5s+5,$$ 
 a contradiction.

\medskip
{\bf Case 2:} $|\cB'|=1$. By the symmetry between $P$ and $Q$, we may assume that $\cB'=\{Q\}$. By the case,
$e(S\cup P,Q\cup R)+e(Q,R)\geq 2(s'+p)+q$. By Theorem~\ref{PQ},
$e(G[S\cup P])\geq 3(s'-r')$. Thus,
$e(G)\geq 2(s'+p)+q+3(s'-r')\geq 5s'+3$, a contradiction.

\medskip
{\bf Case 3:} $|\cB'|=2$. Then $\cB'=\{P,Q\}$. Let $S'=\{v\in S: |N(v)\cap (P\cup Q)|=2\}$ and $S''=S\setminus S'$.
\begin{Claim}\label{spq} No vertex in $P\cup Q$ is a unique neighbor of two non-adjacent vertices in $S$.
\end{Claim}
\begin{proof} Suppose $v\in P$ is the unique neighbor of non-adjacent $u,u'\in S$.
Then we move $u$ to $P$ and $v$ to $S$. Now all conditions (N1)--(N4) still hold, but $u'$
is movable to $P-v+u$ contradicting (N5). This proves the claim.\end{proof}

\begin{Claim}\label{c3}
$e(G[S\cup P\cup Q])\geq 3s'-1.$
\end{Claim}
\begin{proof} First, we show that  
\begin{equation}\label{18}
\parbox{0.9\linewidth}{\em If $P$ has two vertices $x,x'$ movable to $Q$, then each $y\in Q$ that is a unique neighbor of some $z\in S'$ is not movable to $P$.
}
\end{equation}
Indeed, if such $y\in Q$ is a unique neighbor of $z\in S'$ with $N(z)\cap (P\cup Q)=\{y,y'\}$, then $y'\in P$, and we may assume that $y'\neq x'$. So after moving $y$ to $P$ and $x'$ to $Q-y$, we get a new partition satisfying (N1)--(N4) with $z$ movable to $Q-y+x'$, contradicting (N5). This proves~\eqref{18}.

By~\eqref{18} and the symmetry between $P$ and $Q$ we may assume that
\begin{equation}\label{20}
\parbox{0.9\linewidth}{\em for some $y\in Q$, each unique neighbor of a $z\in S'$ in $Q-y$ has a neighbor in
$P$.
}
\end{equation}

We fix $y\in Q$ satisfying~\eqref{20} and do the following discharging. At the start, the {\em charge} $ch(x)$ is $1$ when $x\in E(G[S\cup P\cup Q])\cup \{y\}$ and is $0$ when $x\in (S\cup P\cup Q)\setminus\{y\}$. Thus 
$\sum_{x\in V(G)\cup E(G)}ch(x)=e(G[S\cup P\cup Q])+1$. Now,

Step 1: each edge $uv\in E(S, P\cup Q)$ gives $1$ to its  end in {$S$};

Step 2: each edge $uv\in E( P, Q)$ gives $1$ to its  end in {$Q$};

Step 3: each edge $uv\in E(G[S])$ gives $1/2$ to each of its  ends;

Step 4: each $u\in  Q$ shares its current charge equally among the vertices in $S$ for which $u$ is the unique neighbor in $Q$.

Let $ch_i(x)$ denote the  charge of $x\in V(G)\cup E(G)$ after Step $i$. 
By the rules, $ch_4(x)\geq 0$ for each $x\in E(G[S\cup P\cup Q])\cup P\cup Q$, and so
$\sum_{v\in S}ch_4(v)\leq e(G[S\cup P\cup Q])+1$.
We claim that
 \begin{equation}\label{19}
\mbox { $ch_4(v)\geq 3$ for every $v\in S$.
}
\end{equation}
Indeed, if $v\in S''$, then already $ch_2(v)\geq 3$. If $v\in S'$, let $v'$ be the unique neighbor of $v$ in $Q$.  Now, if $v$ has no neighbors in $S$, then by~\eqref{20} and Claim~\ref{spq}, $v$ gets at least $1$ from $v'$ on Step 4, thus $ch_4(v)\ge 2+1=3.$ Finally, if  $v$ has  neighbors in $S$, then since $d(v)\leq 4$, there is only one such neighbor, say $v''$.
In this case, $v$ may receive only $1/2$ from $v'$, but it also gets $1/2$ from the edge $vv''$ on Step 3. This proves~\eqref{19}.

Since~\eqref{19} implies $e(G[S\cup P\cup Q])\geq 3s'-1$,  the claim holds.\end{proof}

\medskip
Now by Claim~\ref{c3}, \eqref{15'} and \eqref{12'},
$$e(G)\geq (3s'-1)+(s'+p+q)=4s'-1 +p+q\geq 4s'+1+2r'>5s'
$$
This contradiction finishes the proof.

\section{Helpful theorems}\label{help}

\begin{thm}\cite{HajnalSzemeredi1970}\label{equitable}
    Every graph with maximum degree at most $\ell$ has an equitable $(\ell+1)$-coloring.
\end{thm}

\begin{thm}\label{PQPQR}

Let $2\le \ell\le 3$ and $p_1\ge p_2\ge \cdots\ge p_\ell\ge 2$. Define $c=\max\{ i |\ p_i=p_1, 1\le i\le \ell\}-1$. 

(1) If there is an integer $n_3\ge p_1+p_2+\cdots+p_\ell$ such that $K_c\vee T_{n_3-c,p_1-1}\in \EX(n_3, K_{p_1}\cup\cdots\cup K_{p_\ell})$, then $K_c\vee T_{n-c,p_1-1}\in \EX(n, K_{p_1}\cup\cdots\cup K_{p_\ell})$ for all $n\ge n_3$;

(2) If there is an integer $n_4\ge p_1+p_2+\cdots+p_\ell$ such that $\EX(n_4, K_{p_1}\cup\cdots\cup K_{p_\ell})=\{K_c\vee T_{n_4-c,p_1-1}\}$, then $\EX(n, K_{p_1}\cup\cdots\cup K_{p_\ell})=\{K_c\vee T_{n-c,p_1-1}\}$ for all $n\ge n_4$.
\end{thm}

\begin{proof}  Note that $n_3\le n_4$. We use induction on $n$. Suppose Part 1) holds for all $n_3\leq n'<n$ and let $G\in \EX(n, K_{p_1}\cup\cdots\cup K_{p_\ell})$. If $G$ is $K_{p_1}$-free, then $c=0$, and by Theorem~\ref{Turan}, $G=T_{n,p_1-1}$. Thus, we are done. Assume that $G$ contains an induced $K_{p_1}$.

    \begin{Claim}\label{claim7.1}
    $\delta(G)\le n-1-\lfloor\frac{n-c-1}{p_1-1}\rfloor$.
    \end{Claim}
    
    \begin{proof} If $\delta(G)\ge n-\lfloor\frac{n-c-1}{p_1-1}\rfloor$, then $\Delta(\ol{G})\le \lfloor\frac{n-c-1}{p_1-1}\rfloor-1$.

   {\bf Case 1:} $n\le\ell(p_1-1)+c$. Then $\Delta(\ol{G})\le\ell-2\le 1$. Let $\ol{G}=x_1K_2\cup y_1K_1$. We show:
    \begin{equation}\label{7.2}
    \parbox{0.9\linewidth}{\em For $1\le k\le\ell-1$, if $\ol{G}$ contains disjoint induced $\ol{K}_{p_1},\ldots,\ol{K}_{p_{k}}$, then $\ol{G}$ contains disjoint induced $\ol{K}_{p_1},\ldots,\ol{K}_{p_{k}},\ol{K}_{p_{k+1}}$.}
    \end{equation}
    Take pairwise disjoint independent sets $S_1,\ldots,S_k$ with $|S_i|=p_i$ for each $1\le i\le k$ such that the number of edges of $\ol{G}$ incident with $S_1\cup\cdots\cup S_k$ is as large as possible. If $x_1\le p_1+\cdots+p_k$, then $\alpha(\ol{G}-S_1-\cdots-S_k)=n-p_1-\cdots-p_k\ge p_{k+1}$. If $x_1> p_1+\cdots+p_k$, then $\alpha(\ol{G}-S_1-\cdots-S_k)\ge p_1+\cdots+p_k\ge p_{k+1}$. Thus $\alpha(\ol{G}-S_1-\cdots-S_k)\ge p_{k+1}$. Then there is an independent set $S_{k+1}$ of size $p_{k+1}$. This proves~\eqref{7.2}.
    It follows that $\ol{G}$ contains disjoint induced $\ol{K}_{p_1},\ldots,\ol{K}_{p_\ell}$. Thus $G$ contains $K_{p_1}\cup\cdots\cup K_{p_\ell}$, a contradiction.

   {\bf Case 2:}    $n\ge \ell(p_1-1)+c+1$. Then $\lfloor\frac{n-c-1}{p_1-1}\rfloor\ge\ell$. By Theorem~\ref{equitable}, $\ol{G}$ has an equitable coloring with color classes $V_1,V_2,\dots,V_{\lfloor\frac{n-c-1}{p_1-1}\rfloor}$, where $|V_1|\ge |V_2|\ge \dots\ge |V_{\lfloor\frac{n-c-1}{p_1-1}\rfloor}|$. Since $G$ is $K_{p_1}\cup\cdots\cup K_{p_\ell}$-free, $|V_{c+1}|\le p_1-1$. Then $|V_j|\le p_1$ for $1\le j\le c$ and $|V_j|\le p_1-1$ for $c+1\le j\le \lfloor\frac{n-c-1}{p_1-1}\rfloor$. Hence $$n=\sum_{j=1}^{\lfloor\frac{n-c-1}{p_1-1}\rfloor} |V_j|\le c p_1+(\lfloor\frac{n-c-1}{p_1-1}\rfloor-c)(p_1-1)<n,$$ a contradiction.\end{proof}

    By Claim~\ref{claim7.1}, there is a vertex $v\in V(G)$ with $d(v)=\delta(G)\le n-1-\lfloor\frac{n-c-1}{p_1-1}\rfloor$. Since $e(G)\ge e(K_c\vee T_{n-c,p_1-1})$ and $$e(G)=e(G-v)+d(v)\le e(K_c\vee T_{n-c-1,p_1-1})+n-1-\lfloor\frac{n-c-1}{p_1-1}\rfloor=e(K_c\vee T_{n-c,p_1-1}),$$ we conclude that $e(G)=e(K_c\vee T_{n-c,p_1-1})$, $d(v)=n-1-\lfloor\frac{n-c-1}{p_1-1}\rfloor$ and $e(G-v)=e(K_c\vee T_{n-c-1,p_1-1})$. This proves Part 1).
    
    Suppose now that Part 2) holds for all $n_4\leq n'<n$.
    Then $G-v=K_c\vee T_{n-c-1,p_1-1}$ and hence $V(G)-v$ can be divided into $p_1$ subsets $V',V_1,V_2,\ldots,V_{p_1-1}$ so that $G[V']=K_c$, each of $V_1,\ldots,V_{p_1-1}$ is an independent set and 
    $\lceil\frac{n-c-1}{p_1-1}\rceil\geq |V_1|\geq |V_2|\geq \ldots\geq |V_{p_1-1}|\geq\lfloor\frac{n-c-1}{p_1-1}\rfloor.$
    
    If $N(v)\cap V_j\not=\emptyset$ for every $1\leq j\leq p_1-1$, let $v_j\in N(v)\cap V_j$, $U_1=\{v,v_1,v_2,\ldots,v_{p_1-1}\}$, and $G'=G-U_1$. Then $G[U_1]=K_{p_1}$ and $G'=K_c\vee T_{n-p_1-c,p_1-1}$. Since $n-p_1\geq p_2+\cdots+p_\ell$, $G'$ contains $K_{p_2}\cup\cdots\cup K_{p_\ell}$.
    Thus $G$ contains a copy of $K_{p_1}\cup\cdots\cup K_{p_\ell}$, a contradiction. 

    Hence there is an index $j_0\in\{1,\dots,p_1-1\}$ such that $N(v)\cap V_{j_0}=\emptyset$. Since $$n-1-\left\lfloor\frac{n-c-1}{p_1-1}\right\rfloor=d(v)\le n-1-|V_{j_0}|\le n-1-\left\lfloor\frac{n-c-1}{p_1-1}\right\rfloor,$$ $N(v)=V(G)\setminus(V_{j_0}\cup\{v\})$, $|V_{j_0}|=\lfloor\frac{n-c-1}{p_1-1}\rfloor$ and $G=K_c\vee T_{n-c,p_1-1}$. This proves Part 2). 
\end{proof}

\begin{proof}[\bf Proof of Theorem~\ref{thm11}] Let $G\in\CEX(p+q+r,K_p\cup K_q\cup K_r)$. Observe that { for every  $t\geq 1$, each independent set of  size $t$ in $K_4\cup\ol{K}_{p+q+r-4}$ contains at least $t-1$ vertices from $\ol{K}_{p+q+r-4}$}. Since $(p-1)+(q-1)+(r-1)-(p+q+r-4)=1$, $K_4\cup\ol{K}_{p+q+r-4}$ contains no three disjoint independent sets of sizes $p,q$ and $r$. By definition, $\ol{T}_{n,p-1}$ has no induced $\ol{K}_{p}$. Thus $e(G)\le\min\{6, \ol{t}_{n,p-1}\}$.

\begin{Claim}\label{7.1}
    If $G$ contains an induced $\ol{K}_p$ and $e(G)\le 6$, then one of the following holds:

    1) $G=K_4\cup\ol{K}_{n-4}$;
    
    2) $(p,q,r)=(3,2,2)$ and $G=K_{1,6}$.
\end{Claim}

\begin{proof}
    By $e(G)\le 6<9\leq 3(s+r)$ and Theorem~\ref{PQ}, $G$ contains disjoint induced $\ol{K}_{p}$ and $\ol{K}_{q}$. If $e(G)\le 5$, then by Theorem~\ref{PQR}, $G$ contains disjoint induced $\ol{K}_{p},\ol{K}_{q}$ and $\ol{K}_{r}$, a contradiction. Suppose $e(G)=6$. 
    
    If $\Delta(G)=6$, then $G=K_{1,6}\cup \ol{K}_{p+q+r-7}$. If $p+q+r-7\not= 0$, then $K_{1,6}\cup\ol{K}_{p+q+r-7}$ contains disjoint induced $\ol{K}_{p},\ol{K}_{q}$ and $\ol{K}_{r}$, a contradiction. If $p+q+r-7=0$, then by $p\ge q\ge r\ge2$, $(p,q,r)=(3,2,2)$. So, in this case, Part 2) holds.

    If $\Delta(G)\le 5$, let disjoint sets $P,Q\subseteq V(G)$ be such that $G[P]=\ol{K}_p$ and $G[Q]=\ol{K}_q$. Choose $P$ and $Q$ so that \begin{enumerate}[label=\textup{(M\arabic*)}, ref=\textup{M\arabic*}]
    \item \label{PQR3} {\it $G-P-Q$ has the fewest edges,}
    \item \label{PQR4} {\it modulo~\eqref{PQR3}, $e[P,V(G)\setminus P]$ is maximum,}
    \end{enumerate}
    If $|E(P,V(G)\setminus P)|\ge 4$ or $|E(Q,V(G)\setminus Q)|\ge 4$, then $e(G-P)\le 2$ or  $e(G-Q)\le 2$. { By Theorem~\ref{PQ},} in these cases $G-P$ contains disjoint induced $\ol{K}_{q}$ and $\ol{K}_{r}$ or $G-Q$ contains disjoint induced $\ol{K}_{p}$ and $\ol{K}_{r}$, a contradiction.

    Suppose $|E(P,V(G)\setminus P)|\le 3$ and $|E(Q,V(G)\setminus Q)|\le 3$. If $G-P-Q$ has no edges, then $G$ contains disjoint induced $\ol{K}_{p},\ol{K}_{q}$ and $\ol{K}_{r}$, a contradiction. Let $R=V(G)-P-Q$ and $xy\in E(G)$ for some $x,y\in R$. Similarly to~\eqref{01}, we have the following properties.
    \begin{equation}\label{03} \parbox{0.9\linewidth}{\it (A) If $x$ (or $y$) has no neighbors  in some $J\in \{P,Q\}$, then each $w\in J$ has a neighbor in $R$; and\\
    (B) If $x$ (or $y$) has exactly one neighbor $w$  in some $J\in \{P,Q\}$, then $w$ has another neighbor in $R$; and\\
    (C) For each $J\in \{P,Q\}$, $e[J,R]\geq 2$ with equality only if $|J|=2$ or some $w\in J$ is adjacent to both $x$ and $y$.}
    \end{equation}
    Recall that $e(G)=6$. Since $p\ge q\ge r\ge 2$ and $(p,q,r)\not=(2,2,2)$, we have $p\ge 3$. { By~\eqref{03}(C), $e(G[R])\le 2$. If $e(G[R])=2$, then $e[P,R]=2$ and $e[Q,R]=2$. Also by~\eqref{03}(C), there is a vertex $w_r\in R-\{x,y\}$ with degree one in $G[R]$ that has no neighbor in $P$. We can move $w_r$ to $P$ and an isolated vertex in $P$ to $R$, contradicting~\eqref{PQR3}. Thus $e(G[R])=1$.} 
    
    If $e[P,Q]=0$, then there is a set $J\in \{P,Q\}$ such that $e[J,R]=3$. If $e[P,R]=3$, then $e[Q,R]=2$. If $|Q|=2$, let $Q=\{w_{q_1},w_{q_2}\}$ and $xw_{q_1},xw_{q_2}\in E(G)$. Since $\Delta(G)\le 5$, either there is a vertex $w_p\in P$ such that $N(w_p)=\emptyset$ or $|P|=3$ and each vertex in $P$ is adjacent to $y$. If the former holds, then we move $w_p$ to $Q$ and $w_{q_1}$ to $P$, obtaining a new partition which increases $e[P,V(G)\setminus P]$, contradicting~\eqref{PQR4}. If the latter holds, let $P=\{w_{p_1},w_{p_2},w_{p_3}\}$. Then the sets $P'=\{w_{p_1},w_{p_2},x\}$, $Q'=\{w_{q_1},w_{p_3}\}$ and $R'=\{w_{q_2},y\}$ are independent, a contradiction. If there is a vertex $w_q\in Q$ such that $xw_q, yw_q\in E(G)$, then by $e[P,R]=3$, there is a vertex $w_p'\in P$ such that $d(w_p')\le 1$. Thus we can switch $w_p'$ and $w_q$ and increase $e[P,V(G)\setminus P]$, contradicting~\eqref{PQR4}.

    If $e[Q,R]=3$, then $e[P,R]=2$ and there is a vertex $z_q\in Q$ such that $|N(z_q)|\ge 1$. Since $p\ge 3$, there is a vertex $z_p\in P$ such that $N(z_p)=\emptyset$. We can switch $z_p$ and $z_q$ and increase $e[P,V(G)\setminus P]$, contradicting~\eqref{PQR4}.

    If $e[P,Q]=1$, then by~\eqref{03}(C), we have $e[J,R]=2$ for each $J\in \{P,Q\}$. Since $|P|\ge 3$, there is a vertex $u_P\in P$ adjacent to both $x$ and $y$. If $|Q|=2$, then we may assume that $Q=\{u_{q_1},u_{q_2}\}$, $xu_{q_1},xu_{q_2}\in E(G)$ and $N(u_{q_1})\cap P=\emptyset$. Then we move $u_{q_1}$ to $P$ and an isolated vertex in $P$ to $Q$, obtaining a new partition which increases $e[P,V(G)\setminus P]$, contradicting~\eqref{PQR4}.

    Suppose that there is a vertex $u_Q\in Q$ adjacent to both $x$ and $y$. If $u_Pu_Q\not\in E(G)$ and $|N(u_Q)\cap P|=0$, then we move $u_Q$ to $P$ and an isolated vertex in $P$ to $Q$, obtaining a new partition which increases $e[P,V(G)\setminus P]$, contradicting~\eqref{PQR4}. If $u_Pu_Q\not\in E(G)$ and $|N(u_Q)\cap P|=1$, then we can switch $u_Q$ with the neighbor of $u_Q$ in $P$ and increase $e[P,V(G)\setminus P]$, contradicting~\eqref{PQR4}. Thus $u_Pu_Q\in E(G)$. Hence $G=K_4\cup\ol{K}_{n-4}$. 
\end{proof}

If $G$ contains no induced $\ol{K}_p$, then by Theorem~\ref{Turan}, $G=\ol{T}_{n,p-1}$. If $G$ contains an induced $\ol{K}_p$, then by Claim~\ref{7.1}, $G=K_4\cup\ol{K}_{n-4}$ or $(p,q,r)=(3,2,2)$ and $G=K_{1,6}$. Note that $\ol{t}_{n,p-1}=6$ for $(p,q,r)=(5,2,2)$ or $p\ge 7$ and $(q,r)=(3,2)$, $\ol{t}_{n,p-1}=5$ for $p\ge 6$ and $q=r=2$, and $\ol{t}_{n,p-1}>6$ otherwise. Then we are done.\end{proof}

\section{Proof of Theorem~\ref{pqr2}}\label{main}

Recall that { $$c=c_q+c_r,$$ where $p\ge q\ge r$, $c_q=\left\{\begin{array}{lll} 0, & \mbox{if $p>q$,}\\ 
1,  & \mbox{if $p=q$,}\end{array}\right.$ and $c_r=\left\{\begin{array}{lll} 0, & \mbox{if $p>r$,}\\ 
1,  & \mbox{if $p=r$.}\end{array}\right.$}

\begin{lemma}\label{lemma8.1}
    For any $z\leq p-1$, every graph in ${\KK_{3,4}(n,z)}\cup\{K_{3,4}^*(n)\}$  contains no disjoint induced $\ol{K}_p$ and $\ol{K}_q$.
\end{lemma}

{
\begin{proof} Let $z\leq p-1$.
    If $G\in{ \KK_{3,4}(n,z)}$, write $G=K_{3,4}(n,x,y)$, where $x+2y=n-p-q+1$, $3x+4y\le n$ and { $x+y\le z$}.     Each independent set of size $t\in \{p,q\}$ in $G$ contains at least $t-x-y$ vertices from $(n-3x-4y)K_1$. Moreover, $(p-x-y)+(q-x-y)-(n-3x-4y)=p+q+x+2y-n=1$. Thus any   independent $p$-set intersects each independent $q$-set. 

    Suppose $G=K_{3,4}^*(n)=(2p+2q-4-n)K_3\cup (n-\frac{3(p+q)}{2}+2)K_4\cup K_{1,3}$, where $p=q$. Since $K_{1,3}$ has no two disjoint induced $\ol{K}_2$, $G$ contains no disjoint induced $\ol{K}_p$ and $\ol{K}_q$.
\end{proof}}

\begin{lemma}\label{lemma8.2}
  For any $z\leq p-1$,  no graph in ${ \KK_{5,6}(n,z)}\cup\{K_{5,6}^*(n)\}$ contains disjoint induced $\ol{K}_p,\ol{K}_q$ and $\ol{K}_r$.
\end{lemma}

{
\begin{proof}
  Let $z\leq p-1$.   If $G\in{ \KK_{5,6}(n,z)}$, write $G=K_{5,6}(n,x,y)$, where $2x+3y=n-p-q-r+1$, $5x+6y\le n$ and { $x+y\le z$}. 
    Each independent set of size $t\in \{p,q,r\}$  in $G$ contains at least $t-x-y$ vertices from $(n-5x-6y)K_1$. Moreover,$(p-x-y)+(q-x-y)+(r-x-y)-(n-5x-6y)=p+q+r+2x+3y-n=1$. Thus $G$ contains no disjoint induced $\ol{K}_p,\ol{K}_q$ and $\ol{K}_r$.
    
    Suppose now $G=K_{5,6}^*(n)=(2p+2q+2r-6-n)K_5\cup (n-\frac{5(p+q+r)}{3}+4)K_6\cup K_{1,5}$, where $p=q=r$. Since $K_{1,5}$ has no three disjoint induced $\ol{K}_2$, $G$ contains no disjoint induced $\ol{K}_p,\ol{K}_q$ and $\ol{K}_r$.
\end{proof}}

\begin{lemma}\label{3r+3s} If $n\in [p+q, 3(p-1)+c_q]$ and 
\begin{equation}\label{3hold}
   3(n-p-q+1)\leq\ol{t}_{n-c_q,p-1},
\end{equation}
 then
there is a graph $G\in \KK_{3,4}(n,p-1)$ with 
 $e(G)\leq 3(n-p-q+1)$.

Moreover, if $n\in [3(p-1)+c_q, 4(p-1)+c_q]$, then  $3(n-p-q+1)\geq\ol{t}_{n-c_q,p-1}$.
\end{lemma}

\begin{proof}
First suppose that $p+q\leq n\leq 3(p-1)+c_q$
and~\eqref{3hold} holds.
Then every component of $\ol{T}_{n-c_q,p-1}$ has at most
three vertices, so $\ol{t}_{n-c_q,p-1}\leq\frac{2(n-c_q)}{2}=n-c_q$ and
$\ol{t}_{n-c_q,p-1}\leq 3(p-1)$.

Thus $3(n-p-q+1)\leq n-c_q$ and
$n-p-q+1\leq p-1$.
Let $k=\lceil(n-p-q+1)/2\rceil$,
$x=2k-(n-p-q+1)$ and $y=n-p-q+1-k$.
Since $n-p-q+1\geq 1$, we have $x\in\{0,1\}$ and $y\geq 0$.
Moreover,
$$
x+2y=n-p-q+1,\qquad
x+y=k\leq n-p-q+1\leq p-1,
$$
and
$$
3x+4y=n-p-q+1+2k
\leq 3(n-p-q+1)\leq n-c_q\leq n.
$$
Thus $K_{3,4}(n,x,y)\in\KK_{3,4}(n,p-1)$
and $e(K_{3,4}(n,x,y))=3(n-p-q+1)$.

Suppose that $3(p-1)+c_q\leq n\leq 4(p-1)+c_q$.
In this range,
$\ol{t}_{n-c_q,p-1}=3(n-c_q)-6(p-1)$.
By the definition of $c_q$, $q\leq p-1+c_q$, and hence
$$
3(n-p-q+1)\geq 3(n-2(p-1)-c_q)
=\ol{t}_{n-c_q,p-1}.
$$
\end{proof}

\begin{lemma}\label{5s} If $n\in [p+q+r+1, 5(p-1)+c]$ and
\begin{equation}\label{5hold}
   5s\leq \ol{t}_{n-c,p-1}, \end{equation}
 then
there is a graph $G\in \KK_{5,6}(n,p-1)$ with  $e(G)\leq 5s$.

Moreover, if $n\in [5(p-1)+c, 6(p-1)+c]$, then  $5s\geq\ol{t}_{n-c,p-1}$.
\end{lemma}

\begin{proof}
First suppose that $p+q+r+1\leq n\leq 5(p-1)+c$ and~\eqref{5hold} holds.
Then every component of $\ol{T}_{n-c,p-1}$ has at most five vertices, so $\ol{t}_{n-c,p-1}\leq \frac{4(n-c)}{2}=2(n-c)$ and $\ol{t}_{n-c,p-1}\leq 10(p-1)$.

Consequently, $5s\leq 2(n-c)$ and $s\leq 2(p-1)$.
Let $k=\lceil s/3\rceil$, $x=3k-s$ and $y=s-2k$.
Since $s\geq 2$, we have $k\leq s/2$, and therefore $x,y$ are nonnegative integers satisfying
$$
2x+3y=s,\qquad x+y=k\leq p-1,\qquad
5x+6y=s+3k\leq \frac{5s}{2}\leq n-c\leq n.
$$
Thus $K_{5,6}(n,x,y)\in\KK_{5,6}(n,p-1)$ and $e(K_{5,6}(n,x,y))=5s$.

Suppose that $5(p-1)+c\leq n\leq 6(p-1)+c$. In this range, $\ol{t}_{n-c,p-1}=5(n-c)-15(p-1)$. By the definitions of $c$, $$5s=5(n-p-q-r+1)\geq 5(n-3(p-1)-c)=\ol{t}_{n-c,p-1}.
$$
\end{proof}

\bigskip
\begin{proof}[\bf Proof of Theorem~\ref{pqr2}] Let $G$ be an $n$-vertex graph with no disjoint independent sets of sizes $p,q$ and $r$ and with
$e(G)=\cex(n,K_p\cup K_q\cup K_r)$. Since $G$ is extremal,
\begin{equation}\label{crG}
    e(G)\leq \ol{t}_{n-c,p-1}.
\end{equation}
If $G$ does not contain an induced $\ol{K}_p$, then by 
Theorem~\ref{Turan}, $e(G)\geq \ol{t}_{n,p-1}\geq \ol{t}_{n-c,p-1}$. Together with~\eqref{crG}, we get $c=0$ and $e(G)= \ol{t}_{n-c,p-1}$,
so our theorem holds. Hence below we assume $G$ contains  an induced $\ol{K}_p$.

\medskip
{\bf Case 1:}  $G$ does not contain disjoint induced $\ol{K}_p$ and $\ol{K}_q$.
In this case, for any edge $e\in E(G)$, the graph $G-e$ still does not have three disjoint copies of induced $\ol{K}_p$. Thus, by the extremality of $G$, in this case $c\leq 1$. Since $\cex(n,K_p\cup K_q\cup K_r)\leq \cex(n,K_p\cup K_q)$, we conclude that
 $e(G)=\cex(n,K_p\cup K_q)$.
 
 By Theorem~\ref{PQ} and~\eqref{crG}, $3(r+s)\le e(G)\le\ol{t}_{n-c,p-1}$.
If $n\in [p+q+r+1, 4(p-1)+c]$, then by Lemma~\ref{3r+3s}, we have $e(G)= 3(r+s)$.

Assume that $n\geq 4(p-1)+c+1$. Since
$\ol{t}_{4(p-1),p-1}=6(p-1)$ and 
$$3((4(p-1)+c)-p-q+1)=3(3p-4-(q-c)+1)\geq 3(2p-2)=\ol{t}_{4(p-1),p-1},
$$ Theorem~\ref{PQPQR}(1) for $\ell=2$ yields that $e(G)=\cex(n,K_p\cup K_q)=\ol{t}_{n-c,p-1}$.

\medskip
{\bf Case 2:}  $G$ contains disjoint induced $\ol{K}_p$ and $\ol{K}_q$. By Theorem~\ref{PQR} and~\eqref{crG}, $5s\le e(G)\le\ol{t}_{n-c,p-1}$.
If $n\in [p+q+r+1, 5(p-1)+c]$, then by Lemma~\ref{5s}, we have $e(G)=5s$.
 
Assume that $n\geq 5(p-1)+c+1$. Since $\ol{t}_{5(p-1),p-1}=10(p-1)$
and 
$$5((5(p-1)+c)-p-q-r+1)\geq \left\{\begin{array}{lll} 5(4p-5-(q-c)-r+1)\geq 5(2p-2) & \mbox{if $c\leq 1$,}\\ 
5(4p-5-(q-1)-(r-1)+1)\geq 5(2p-2) & \mbox{if $c=2$,}\end{array}\right.
$$ Theorem~\ref{PQPQR}(1)  for $\ell=3$ yields that $e(G)=\cex(n,K_p\cup K_q\cup K_r)=\ol{t}_{n-c,p-1}$.

 Again, our theorem holds. \end{proof}

\section{Proof of Theorem~\ref{EXpqr2}\label{main2}}

\bigskip
\bigskip

Let $\GG_n=\CEX(n,K_p\cup K_q\cup K_r)$.  Then for each $G\in \GG_n$, $e(G)\leq \ol{t}_{n-c,p-1}$.

\begin{Claim}\label{claim9.1} Let $G\in \GG_n$.
    (i) If $G$ contains no $\ol{K}_p$, then $c=0$;\\
    (ii) If $G$ contains $\ol{K}_p$ and no disjoint induced $\ol{K}_p$ and $\ol{K}_q$, then $c\le 1$.
\end{Claim}

\begin{proof}Let $e\in E(G)$. If $G$ contains no $\ol{K}_p$, then $G-e$ does not have two disjoint copies of induced $\ol{K}_p$. Thus $c=0$. If $G$ contains $\ol{K}_p$ and no disjoint induced $\ol{K}_p$ and $\ol{K}_q$, then $G-e$ does not have three disjoint copies of induced $\ol{K}_p$. Hence $c\le 1$. \end{proof}

\begin{Claim}\label{claim9.2} 
    (i)  {   If $n\ge 4(p-1)+c_q$,
     then} $\CEX(n, K_p\cup K_q)=\{\ol{K}_{c_q}\cup\ol{T}_{n-c_q,p-1}\}$;\\
    (ii)  {   If $n\ge 6(p-1)+c-1$,
     then}
    $\CEX(n, K_p\cup K_q\cup K_r)=\{\ol{K}_c\cup\ol{T}_{n-c,p-1}\}$.
\end{Claim}

\begin{proof} Suppose there is $J_{4(p-1)+c_q}\in\CEX(4(p-1)+c_q, K_p\cup K_q)$. 
Note that $e(J_{4(p-1)+c_q})\le\ol{t}_{4(p-1),p-1}$. If $J_{4(p-1)+c_q}$ contains no $\ol{K}_p$, then by Theorem~\ref{Turan}, $c_q=0$ and $J_{4(p-1)+c_q}=\ol{T}_{4(p-1),p-1}$. 

If $J_{4(p-1)+c_q}$ contains $\ol{K}_p$, then by Theorem~\ref{PQ}, $e(J_{4(p-1)+c_q})\ge 3(4(p-1)+c_q-p-q+1)$. Since
$\ol{t}_{4(p-1),p-1}=6(p-1)$ and 
$$3(4(p-1)+c_q-p-q+1)=3(3p-3+c_q-q)\geq 3(2p-2)=\ol{t}_{4(p-1),p-1},
$$we have $3(4(p-1)+c_q-p-q+1)=\ol{t}_{4(p-1),p-1}$. Hence $q=\left\{\begin{array}{lll} p-1, & \mbox{if $c_q=0$,}\\ p, & \mbox{if $c_q=1$.}\end{array}\right.$ By Theorem~\ref{PQ}, $J_{4(p-1)+c_q}=xK_3\cup yK_4\cup\ol{K}_{4(p-1)+c_q-3x-4y}$ or $J_{4(p-1)+c_q}=x'K_3\cup y'K_4\cup K_{1,3}$. If the former holds, then since
$3x+6y=3(4(p-1)+c_q-p-q+1)=6(p-1)$ and $3x+4y\le 4(p-1)+c_q$, we have $x=0$ and $y=p-1$. If $c_q=0$, then $J_{4(p-1)+c_q}=(p-1)K_4$, a contradiction, as $(p-1)K_4$ contains no $\ol{K}_p$. If $c_q=1$, then $J_{4(p-1)+c_q}=K_1\cup (p-1)K_4=\ol{K}_1\cup\ol{T}_{4(p-1),p-1}$. If the latter holds, then since $p=q$, $3x'+6y'+3=3(4(p-1)+c_q-p-q+1)=6(p-1)$ and $3x'+4y'+4=4(p-1)+c_q=4p-3$, we have $x'=-1$ and $y'=p-1$, a contradiction. Thus $\CEX(4(p-1)+c_q, K_p\cup K_q)=\{\ol{K}_{c_q}\cup\ol{T}_{4(p-1),p-1}\}$. Then Theorem~\ref{PQPQR}(2) for $\ell=2$ yields that $\CEX(n, K_p\cup K_q)=\{\ol{K}_{c_q}\cup\ol{T}_{n-c_q,p-1}\}$ for $n\ge 4(p-1)+c_q$. This proves (i).

Suppose that there is $G_{6(p-1)+c-1}\in\CEX(6(p-1)+c-1, K_p\cup K_q\cup K_r)$. Recall that $e(G_{6(p-1)+c-1})\le\ol{t}_{6(p-1)-1,p-1}$. If $G_{6(p-1)+c-1}$ contains no $\ol{K}_p$, then by Theorem~\ref{Turan}, $c=0$ and $G_{6(p-1)+c-1}=\ol{T}_{6(p-1)-1,p-1}$. 

If $G_{6(p-1)+c-1}$ contains $\ol{K}_p$ and no disjoint induced $\ol{K}_p$ and $\ol{K}_q$, then by Claim~\ref{claim9.1}, $c\le 1$. Moreover, $e(G_{6(p-1)+c-1})\ge\cex(6(p-1)+c-1,K_p\cup K_q)$. Since $\cex(n,K_p\cup K_q\cup K_r)\le\cex(n,K_p\cup K_q)$, $e(G_{6(p-1)+c-1})=\cex(6(p-1)+c-1,K_p\cup K_q)$. Thus $G_{6(p-1)+c-1}\in\CEX(6(p-1)+c-1, K_p\cup K_q)$. By (i), $G_{6(p-1)+c-1}=\ol{K}_c\cup\ol{T}_{6(p-1)-1,p-1}$. 

If $G_{6(p-1)+c-1}$ contains disjoint induced $\ol{K}_p$ and $\ol{K}_q$, then by Theorem~\ref{PQR}, $e(G_{6(p-1)+c-1})\ge 5(6(p-1)+c-1-p-q-r+1)$. Since
$\ol{t}_{6(p-1)-1,p-1}=15(p-2)+10=15p-20$ and 
$$5(6(p-1)+c-1-p-q-r+1)=5(5p-6+c-q-r)\geq 5(3p-4)=\ol{t}_{6(p-1)-1,p-1},
$$we have $5(6(p-1)+c-1-p-q-r+1)=\ol{t}_{6(p-1)-1,p-1}$. Hence 
$$\mbox{
$q=\left\{\begin{array}{lll} p-1, & \mbox{if $c=0$,}\\ p, & \mbox{if $c\ge 1$,}\end{array}\right.$ and $r=\left\{\begin{array}{lll} p-1, & \mbox{if $c\le 1$,}\\ p, & \mbox{if $c=2$.}\end{array}\right.$ }
$$
By Theorem~\ref{PQR}, $G_{6(p-1)+c-1}=xK_5\cup yK_6\cup\ol{K}_{6(p-1)+c-1-5x-6y}$ or $G_{6(p-1)+c-1}=x'K_5\cup y'K_6\cup K_{1,5}$. If the former holds, then since
$$\mbox{
$10x+15y=5(6(p-1)+c-1-p-q-r+1)=15p-20$ and $5x+6y\le 6(p-1)+c-1$,}$$
we have $x=1$ and $y=p-2$. If $c\le 1$, then $G_{6(p-1)+c-1}=\ol{K}_c\cup K_5\cup (p-2)K_6$, a contradiction as $K_5\cup (p-2)K_6$ and $\ol{K}_1\cup K_5\cup (p-2)K_6$ contain no disjoint induced $\ol{K}_p$ and $\ol{K}_q$. If $c=2$, then $G_{6(p-1)+c-1}=\ol{K}_2\cup K_5\cup (p-2)K_6=\ol{K}_2\cup\ol{T}_{6(p-1)-1,p-1}$. If the latter holds, then since $p=q=r$, $$10x'+15y'+5=5(6(p-1)+c-1-p-q-r+1)=15p-20$$ and $5x'+6y'+6=6(p-1)+c-1=6p-5$, we have $x'=-1$ and $y'=p-1$, a contradiction. Thus $\CEX(6(p-1)+c-1, K_p\cup K_q\cup K_r)=\{\ol{K}_c\cup\ol{T}_{6(p-1)-1,p-1}\}$. Then Theorem~\ref{PQPQR}(2) for $\ell=3$ yields that $\CEX(n, K_p\cup K_q\cup K_r)=\{\ol{K}_c\cup\ol{T}_{n-c,p-1}\}$ for $n\ge 6(p-1)+c-1$. This proves (ii). \end{proof}

\begin{Claim}\label{claim9.3}
    $\CEX(n, 2K_p)=\left\{\begin{array}{lll} \KK_{3,4}(n), & \mbox{if $2p\le n\le 3p-3$,}\\
    \KK_{3,4}(n)\cup\{K_{3,4}^*(n)\}, & \mbox{if $3p-2\le n\le 4p-4$,}\\ 
    \{{K}_1\cup\overline{T}_{n-1,p-1}\},  & \mbox{$n\ge 4p-3$.}\end{array}\right.$
\end{Claim}

\begin{proof} Let $J_{n}\in\CEX(n, 2K_p)$. Since $e(J_{n})\le\ol{t}_{n-1,p-1}$, by Theorem~\ref{Turan}, $J_{n}$ contains $\ol{K}_p$. By Claim~\ref{claim9.2}(i),
if  $n\ge 4p-3$, then $\CEX(n, 2K_p)=\{\ol{K}_1\cup\overline{T}_{n-1,p-1}\}$. 

Then we consider $2p\le n\le 4p-4$. By Theorem~\ref{PQ}, $J_{n}\in\KK_{3,4}(n)$ or $J_{n}=x'K_3\cup y'K_4\cup K_{1,3}$. If the latter holds, then by $3x'+6y'+3=3(n-2p+1)$ and $3x'+4y'+4=n$, we have $x'=4p-4-n$ and $y'=n-3p+2$, i.e., $J_{n}=K_{3,4}^*(n)$. On the other hand, by Lemma~\ref{lemma8.1}, $\KK_{3,4}(n)\subseteq\CEX(n, 2K_p)$ and $K_{3,4}^*(n)\in\CEX(n, 2K_p)$. 
\end{proof}

\begin{Claim}\label{claim9.4}
    $\CEX(n, 3K_p)=\left\{\begin{array}{lll} \KK_{5,6}(n), & \mbox{if $3p+1\le n\le 5p-5$,}\\
    \KK_{5,6}(n)\cup\{K_{5,6}^*(n)\}, & \mbox{if $5p-4\le n\le 6p-6$,}\\ 
    \{\ol{K}_2\cup\overline{T}_{n-2,p-1}\},  & \mbox{$n\ge 6p-5$.}\end{array}\right.$
\end{Claim}

\begin{proof} Let $G_{n}\in\CEX(n, 3K_p)$ and $e\in E(G_{n})$. If $G_n$ contains no two disjoint induced $\ol{K}_p$, then $G_n-e$ does not have three disjoint induced $\ol{K}_p$, a contradiction. Thus $G_n$ contains two disjoint induced $\ol{K}_p$. By Claim~\ref{claim9.2}(ii), if $n\ge 6p-5$, then $\CEX(n, 3K_p)=\{\ol{K}_2\cup\overline{T}_{n-2,p-1}\}$. 

Then we consider $3p+1\le n\le 6p-6$. By Theorem~\ref{PQR}, $G_{n}\in\KK_{5,6}(n)$ or $G_{n}=x'K_5\cup y'K_6\cup K_{1,5}$. If the latter holds, then since $10x'+15y'+5=5(n-3p+1)$ and $5x'+6y'+6=n$, we have $x'=6p-6-n$ and $y'=n-5p+4$, i.e., $G_{n}=K_{5,6}^*(n)$. By Lemma~\ref{lemma8.2}, $\KK_{5,6}(n)\subseteq\CEX(n, 3K_p)$ and $K_{5,6}^*(n)\in\CEX(n, 3K_p)$. \end{proof}

\medskip

\begin{Claim}\label{claim9.5}
    (i) If $c=0$, then $\mathcal{H}_1(n)=\{G\in\KK_{3,4}(n,q-1): \cex(n,K_p\cup K_q\cup K_r)=3(s+r)\};$\\
    (ii) If $c\le 1$, then $\mathcal{H}_2(n)=\{G\in\KK_{5,6}(n,r-1): \cex(n,K_p\cup K_q\cup K_r)=5s\}.$
\end{Claim}

\begin{proof} To prove (i), suppose that  $c=0$. First, we show that
\begin{equation}\label{10.1}
    \mathcal{H}_1(n)\subseteq\{G\in\KK_{3,4}(n,q-1): \cex(n,K_p\cup K_q\cup K_r)=3(s+r)\}.
\end{equation}
{ If $\mathcal{H}_1(n)=\emptyset$, then~\eqref{10.1} holds.} If $\mathcal{H}_1(n)\not=\emptyset$, let $G_1(n)\in\HH_1(n)$. Since $G_1(n)\in\CEX(n,K_p\cup K_q)$, by Claim~\ref{claim9.2}(i), $n\le 4(p-1)-1$. By Theorem~\ref{pqr2}, $e(G_1(n))=\min\{5s,3(s+r),\ol{t}_{n,p-1}\}$ and by Theorem~\ref{PQ},  $e(G_1(n))\ge 3(s+r)$. Thus $e(G_1(n))=3(s+r)$. By Theorem~\ref{PQ}, $G_1(n)=xK_3\cup yK_4\cup\ol{K}_{n-3x-4y}$. If $x+y\ge q$, then 
$$\alpha(G_1(n))=(x+y)+(n-3x-4y)=n-(x+2y)-(x+y)\le n-(n-p-q+1)-q=p-1,$$ contradicting the fact that $G_1(n)$ contains an induced $\ol{K}_p$.  Thus $x+y\le q-1$. This proves~\eqref{10.1}.

Next, we prove that
\begin{equation}\label{10.2}
    \{G\in\KK_{3,4}(n,q-1): \cex(n,K_p\cup K_q\cup K_r)=3(s+r)\}\subseteq\mathcal{H}_1(n).
\end{equation}
Suppose now that $\cex(n,K_p\cup K_q\cup K_r)=3(s+r)$. If $\KK_{3,4}(n,q-1)=\emptyset$, then~\eqref{10.2} holds. If $\KK_{3,4}(n,q-1)\not=\emptyset$, let $G'=xK_3\cup yK_4\cup\ol{K}_{n-3x-4y}\in\KK_{3,4}(n,q-1)$. Since $x+2y=n-p-q+1$ and $x+y\le q-1$, we have $$\alpha(G')=(x+y)+(n-3x-4y)=n-(x+2y)-(x+y)\ge n-(n-p-q+1)-(q-1)=p.$$ Thus $G'$ contains an induced $\ol{K}_p$. By Lemma~\ref{lemma8.1}, $G'$ has no disjoint induced $\ol{K}_p$ and $\ol{K}_q$. Since $e(G')=3(s+r)$, $G'\in\CEX(n,K_p\cup K_q\cup K_r)$. Thus $G'\in\mathcal{H}_1(n)$. This proves~\eqref{10.2}.

To prove (ii) suppose that $c\le 1$. First, we prove that
\begin{equation}\label{10.3}
    \mathcal{H}_2(n)\subseteq\{G\in\KK_{5,6}(n,r-1): \cex(n,K_p\cup K_q\cup K_r)=5s\}.
\end{equation}
If $\mathcal{H}_2(n)=\emptyset$, then~\eqref{10.3} holds. If $\mathcal{H}_2(n)\not=\emptyset$, let $G_2(n)\in\HH_2(n)$. Since $G_2(n)\in\CEX(n,K_p\cup K_q\cup K_r)$, by Claim~\ref{claim9.2}(ii), $n\le 6(p-1)+c-2$. By Theorem~\ref{PQR}, $e(G_2(n))\ge 5s$. Theorem~\ref{pqr2} and Lemma~\ref{5s} yield $e(G_2(n))\leq 5s$. 
Thus, $e(G_2(n))= 5s$. Then Theorem~\ref{PQR} implies that $G_2(n)=xK_5\cup yK_6\cup\ol{K}_{n-5x-6y}$. If $x+y\ge r$, then the maximum total size of two disjoint independent sets in $G_2(n)$ is 
$$2(x+y)+(n-5x-6y)=n-(2x+3y)-(x+y)\le n-(n-p-q-r+1)-r=p+q-1,$$ contradicting the fact that $G_2(n)$ contains disjoint induced $\ol{K}_p$ and $\ol{K}_q$. Thus $x+y\le r-1$. This proves~\eqref{10.3}.

Next, we prove that
\begin{equation}\label{10.4}
    \{G\in\KK_{5,6}(n,r-1): \cex(n,K_p\cup K_q\cup K_r)=5s\}\subseteq\mathcal{H}_2(n).
\end{equation}
Suppose now that $\cex(n,K_p\cup K_q\cup K_r)=5s$. If $\KK_{5,6}(n,r-1)=\emptyset$, then~\eqref{10.4} holds. If $\KK_{5,6}(n,r-1)\not=\emptyset$, let $G'=xK_5\cup yK_6\cup\ol{K}_{n-5x-6y}\in\KK_{5,6}(n,r-1)$. Since $2x+3y=n-p-q-r+1$ and $x+y\le r-1$, the maximum total size of two disjoint independent sets in $G'$ is 
$$2(x+y)+(n-5x-6y)=n-(2x+3y)-(x+y)\ge n-(n-p-q-r+1)-(r-1)=p+q.$$ Thus $G'$ contains disjoint induced $\ol{K}_p$ and $\ol{K}_q$. By Lemma~\ref{lemma8.2}, $G'$ has no disjoint induced $\ol{K}_p$, $\ol{K}_q$ and $\ol{K}_r$. Since $e(G')=5s$, $G'\in\CEX(n,K_p\cup K_q\cup K_r)$. Thus $G'\in\mathcal{H}_2(n)$. This proves~\eqref{10.4}.
\end{proof}

\medskip
{ Now we are ready to finish the proof of the theorem.}

\medskip
{\bf Case 1:} $c=0$. Since $\HH_0(n)\subseteq\CEX(n, K_p)$, by Theorem~\ref{Turan}, $\mathcal{H}_0(n)=\{\overline{T}_{n,p-1}: \cex(n,K_p\cup K_q\cup K_r)=\ol{t}_{n,p-1}\}$. Together with Claim~\ref{claim9.5}, we are done.

\medskip
{\bf Case 2:} $c=1$. By Claim~\ref{claim9.1}(i), $\mathcal{H}_0(n)=\emptyset$. Note that $\HH_1(n)\subseteq\CEX(n, 2K_p)$. By Claim~\ref{claim9.3}, $$\begin{aligned}
\mathcal{H}_1(n)=&\{G\in\KK_{3,4}(n): \cex(n,K_p\cup K_q\cup K_r)=3(s+r)\}\\
&\cup\{K_{3,4}^*(n): \cex(n,K_p\cup K_q\cup K_r)=3(s+r) \mbox{ and }\tfrac{3(p+q)}{2}-2\le n\le 2p+2q-4\}\\
&\cup\{\ol{K}_1\cup\overline{T}_{n-1,p-1}: \cex(n,K_p\cup K_q\cup K_r)=\ol{t}_{n-1,p-1}\}.
\end{aligned}$$By Claim~\ref{claim9.5}(ii), we are done.

\medskip
{\bf Case 3:} $c=2$. In this case, $r=q=p$. 
By Claim~\ref{claim9.1}, $\mathcal{H}_0(n)=\emptyset$ and $\mathcal{H}_1(n)=\emptyset$. Since $\HH_2(n)=\CEX(n, 3K_p)$, by Claim~\ref{claim9.4}, $$\mathcal{H}_2(n)=\left\{\begin{array}{lll} \KK_{5,6}(n), & \mbox{if $n\le 5p-5$,}\\ \KK_{5,6}(n)\cup\{K_{5,6}^*(n)\}, & \mbox{if $5p-4\le n\le 6p-6$,}\\ \{\ol{K}_2\cup\overline{T}_{n-2,p-1}\},  & \mbox{$n\ge 6p-5$.}\end{array}\right.$$\qed

\section{Concluding remarks}\label{remarks}

We think that the following claims are likely to be true.

\begin{Conjecture} \label{conj1}
\label{PQRC} Let $k_1,k_2,\ldots,k_{\ell-1}\geq  k_\ell\geq 2$, $n\geq k_1+k_2+\cdots+k_\ell$ and $s=n-k_1-k_2-\cdots-k_\ell+1$. If an $n$-vertex graph $G$ contains  vertex disjoint induced $\ol{K}_{k_1},\ldots,\ol{K}_{k_{\ell-1}}$ and 
$e(G)<(2\ell-1)s$, then $G$ contains  vertex disjoint  induced $\ol{K}_{k_1},\ldots,\ol{K}_{k_\ell}$.
\end{Conjecture}

{ \begin{Conjecture} 
\label{PQRC2} Let $\ell\ge 4$ and $k_1,k_2,\ldots, k_\ell$ be positive integers with $k_1\ge k_2\ge \ldots\ge k_{\ell-1}\geq  k_\ell\ge 2$. Define $t=\max\{ i |\ k_i=k_1, 1\le i\le \ell\}$. 

(1) If there is an integer $n_3\ge k_1+k_2+\cdots+k_\ell$ such that $K_{t-1}\vee T_{n_3-t+1,k_1-1}\in \EX(n_3, K_{k_1}\cup\cdots\cup K_{k_\ell})$, then $K_{t-1}\vee T_{n-t+1,k_1-1}\in \EX(n, K_{k_1}\cup\cdots\cup K_{k_\ell})$ for all $n\ge n_3$;

(2) If there is an integer $n_4\ge k_1+k_2+\cdots+k_\ell$ such that $\EX(n_4, K_{k_1}\cup\cdots\cup K_{k_\ell})=\{K_{t-1}\vee T_{n_4-t+1,k_1-1}\}$, then $\EX(n, K_{k_1}\cup\cdots\cup K_{k_\ell})=\{K_{t-1}\vee T_{n-t+1,k_1-1}\}$ for all $n\ge n_4$.
\end{Conjecture}}

\section{Acknowledgment}
Part of the work was carried out while the third author was a PhD student at the University of Illinois Urbana-Champaign and was supported in part by NSF RTG Grant DMS-1937241.

\makeatletter
\renewcommand{\@biblabel}[1]{#1.}
\makeatother
\providecommand{\bysame}{\leavevmode\hbox to3em{\hrulefill}\thinspace}

\end{document}